\documentclass[a4paper, twoside, 11pt]{amsart}

\pdfoutput=1 

\usepackage[T1]{fontenc}
\usepackage[utf8]{inputenc}
\usepackage[english]{babel}

\usepackage{lmodern}
\usepackage{amssymb,amsmath,amsthm,amsfonts,mathtools}
\usepackage{dsfont}
\usepackage[hidelinks]{hyperref}
\usepackage{natbib}

\RequirePackage{mathrsfs}

\usepackage{enumitem}

\usepackage[font=small, margin=40pt,figureposition=below]{caption}

\usepackage{tikz-cd}

\newtheorem{theorem}{Theorem}[section]
\newtheorem*{theorem*}{Theorem}
\newtheorem{corollary}[theorem]{Corollary}
\newtheorem{lemma}[theorem]{Lemma}
\newtheorem{proposition}[theorem]{Proposition}
\theoremstyle{definition}
\newtheorem{definition}[theorem]{Definition}
\newtheorem{example}[theorem]{Example}
\newtheorem{remark}[theorem]{Remark}
\newtheorem{conjecture}[theorem]{Conjecture}
\newtheorem{question}[theorem]{Question}

\setlist[enumerate, 1]{label = (\roman*), ref = \roman*}

\newcommand{\CA}{{\mathcal{A}}}
\newcommand{\CO}{{\mathcal{O}}}
\newcommand{\CR}{{\mathcal{R}}}
\newcommand{\cA}{\mathcal{A}}
\newcommand{\cC}{\mathcal{C}}
\newcommand{\cO}{\mathcal{O}}
\newcommand{\cP}{\mathcal{P}}
\newcommand{\cI}{\mathcal{I}}

\newcommand{\bc}{\boldsymbol{c}}
\newcommand{\bx}{{\boldsymbol{x}}}
\newcommand{\RA}{{\mathbf{A}}}

\newcommand{\bX}{{\boldsymbol{X}}}

\newcommand{\mC}{\mathfrak{C}}
\newcommand{\mD}{\mathfrak{D}}

\newcommand{\mT}{\mathfrak{T}}
\newcommand{\mU}{\mathfrak{U}}
\newcommand{\mV}{\mathfrak{V}}

\newcommand{\mX}{\mathfrak{X}}
\newcommand{\mY}{\mathfrak{Y}}
\newcommand{\fo}{\mathfrak{o}}
\newcommand{\fO}{\mathfrak{O}}
\newcommand{\fp}{\mathfrak{p}}

\newcommand{\dA}{\mathbb{A}}
\newcommand{\CC}{\mathbb{C}}
\newcommand{\FF}{\mathbb{F}}
\newcommand{\GG}{\mathbb{G}}
\newcommand{\NN}{\mathbb{N}}
\newcommand{\PP}{\mathbb{P}}
\newcommand{\QQ}{\mathbb{Q}}
\newcommand{\RR}{\mathbb{R}}
\newcommand{\ZZ}{\mathbb{Z}}

\newcommand{\Gm}{\mathbb{G}_\mathrm{m}}

\newcommand{\Gmok}{\mathbb{G}_{\mathrm{m},\mathfrak{o}}}
\newcommand{\GmO}{\mathbb{G}_{\mathrm{m},\mathfrak{O}}}

\DeclareMathOperator{\cl}{cl}
\DeclareMathOperator{\Frac}{Frac}
\DeclareMathOperator{\Gal}{Gal}
\DeclareMathOperator{\im}{im}
\DeclareMathOperator{\mult}{mult}
\DeclareMathOperator{\Pic}{Pic}
\DeclareMathOperator{\rk}{rk}
\DeclareMathOperator{\Spec}{Spec}

\newcommand{\norm}[1]{\left\lVert#1\right\rVert}

\newcommand{\red}{\mathrm{red}}
\newcommand{\et}{\textrm{ét}}
\newcommand{\NS}{\mathrm{NS}}

\begin{document}
	\title{Integral Diophantine approximation on varieties}
	\author{Zhizhong Huang}
	\address{State Key Laboratory of Mathematical Sciences, Academy of Mathematics and Systems Science, Chinese Academy of Sciences, Beijing 100190, China}
	\email{zhizhong.huang@yahoo.com}

	\author{Florian Wilsch}
	\address{Mathematisches Institut, Georg-August-Universität Göttingen, Bunsenstr.\ 3--5, 37073 Göttingen, Germany}
	\email{florian.wilsch@mathematik.uni-goettingen.de}

	\date{May 7, 2026}
	
	\begin{abstract}
	We study the local behavior of integral points on log pairs near a fixed rational point in the boundary by means of an integral approximation constant.
    In light of Siegel's theorem about integral points on curves and McKinnon's conjecture on rational approximation constants, we conjecture that integral points that are close to the fixed point in archimedean topology should lie on certain rational curves with at most two points at infinity on weakly log Fano varieties. We verify this conjecture for a number of examples.
	\end{abstract}
	
	\maketitle
	
	\tableofcontents
	
\section{Introduction}
Solving polynomial equations with integer coefficients by integers lies at the heart of Diophantine geometry. This problem comes in many qualitative and quantitative variants about the existence, number, and distribution of solutions. The goal of this article is to address the following local aspect: how well can one approximate a given point by integral solutions, that is, how quickly can the distance between integral points and a fixed point decrease while keeping their ``complexity'' as small as possible?

\subsection{Approximation constants}
On projective varieties, rational and integral points coincide, and as the rationals are dense in the reals (and more generally, a number field is dense in any of its completions), rational points lend themselves readily to approximation problems; they are thus the original setting for \emph{Diophantine approximation}, asking about rational numbers close to real points. Concretely, given a real number $x$, what is the supremum $\alpha$ of exponents $a>0$ such that
\begin{equation*}
	\left|x-\frac{p}{q}\right| \le \frac{1}{|q|^a}
\end{equation*}
has infinitely many solutions with coprime $(p,q)\in \ZZ^2$? If $x\in \QQ$, then it is easy to see that $\alpha=1$. For irrational algebraic $x$, this question has been famously answered by K.~Roth \cite{Roth}: $\alpha =2$. 

McKinnon~\cite{McK} has proposed a geometric variant of this problem, which he studied together with M.~Roth~\cite{McK-Roth}:
instead of an algebraic number, $x$ becomes a geometric point on a projective variety $X$ over a number field $k$; the rational number $p/q$ becomes a rational point $y$ on $X$, the difference $|x-p/q|$ becomes an appropriately constructed distance $d_v(x,y)$ for a fixed place $v$, and $|q|$ becomes the height $H_L(y)$, where $H_L$ is a height function associated with an ample line bundle $L$ on $X$, which is a measure of the ``complexity'' of rational points. Given such $x\in X(\overline{k})$ and an ample line bundle $L$, the \emph{approximation constant} $\alpha_v(x;L)$ is the infimum of positive real numbers $a$ such that
\begin{equation*}
	d_v(x,y)^a \ll \frac{1}{H_L(y)}
\end{equation*}
has infinitely many solutions with $y\in X(k)$; it measures how well $x$ can be approximated by rational points. Putting the exponent on the distance makes this homogeneous in the line bundle $L$ among other similarities with \emph{Seshadri constants} \cite[\S3]{McK-Roth}, and a smaller approximation constant means that a point is easier to approximate.
K.~Roth's Theorem can then be interpreted as giving the answer for $X=\PP^1$: it is $1$ for rational points, and $1/2$ for algebraic points defined over $\overline{k}\cap k_v$.
McKinnon put forward a conjecture \cite[Conj.~2.7]{McK} on these constants, predicting that whenever $x\in X(k)$ and the approximation constant $\alpha_v(x;L)$ is finite,
a sequence of ``best approximations''~\cite[Def.~2.5]{McK} can be found along a rational curve $C$;
that is, there is a sequence $(y_n)_{n\in \NN}\subseteq C(k)$ with $d_v(x,y_n)\to 0$ and for any $\varepsilon>0$,
\begin{equation*}
	\frac{1}{H_L(y_n)^{1+\epsilon}} \ll_\epsilon d_v(x,y_n)^{\alpha_v(x;L)} \ll \frac{1}{H_L(y_n)^{1-\varepsilon}}.
\end{equation*}
Evidence for this conjecture is underpinned by the fact~\cite[\S\,7.3~2nd Thm, \S\,7.4]{serre-mordell} going back to Siegel~\cite{Siegel} that the approximation constant along any curve of positive genus is infinite.

\subsection{Integral points and approximations} In the context of the classical variant of \emph{Manin's problem} on rational points on Fano varieties $X$~\cite{MR974910}, Peyre's equidistribution principle~\cite{Peyre} involves an (absolutely continuous) \emph{Tamagawa measure} on $X(k_v)$ for each place $v$, with respect to which the set $X(k)$ is equidistributed. In particular, there is an abundance of rational points that can approximate other points whenever this principle holds.
On the other hand, integral points on an affine variety $U\subseteq \dA_\ZZ^n$ lie within the integer lattice $\ZZ^n$, and hence clearly cannot approximate any point in $U$ with respect to the real distance.
Work on the integral variant of Manin's problem, as initiated by Chambert-Loir and Tschinkel \cite{C-L-Tsch} asking about the distribution of integral points of bounded height on a log Fano pair $(X,D)$, suggests that this is the wrong point of view, however, and predicts instead an equidistribution against a certain \emph{residue measure} supported on the boundary $D(k_v)$ as soon as $v$ is archimedean. As a basic example, one may regard the lattice $\ZZ^n=\dA^n(\ZZ)$ inside $\PP^n(\RR)$; this way, it ceases to be discrete, having accumulation points on the hyperplane $H\cong\PP^{n-1}$ at infinity and in fact becomes equidistributed with respect to a measure supported on $H$.
It is thus natural to investigate how well points on the boundary $D$ can be approximated by integral points.
We shall study this question by defining an integral variant of McKinnon's approximation constant. We conjecture that this constant is still governed by rational curves under suitable hypotheses (Conjecture~\ref{conj:integralmckinnon}).

A necessary condition for a rational curve to contribute to the (integral) approximation constant is that it contains infinitely many (integral) points. Siegel's classical theorem translates this into geometric properties of the curve.

\begin{theorem}[Siegel (1929) \cite{Siegel}]\label{thm:Siegel}
	Let $C$ be a smooth, proper curve over a number field $k$ with ring of integers $\fo$, and let $D\subset C$ be a zero-dimensional reduced subscheme. 
	Then there are only finitely many integral points on any $\fo$-model of $C\setminus D$, except possibly if $g(C)=0$ and $\deg(D)\le 2$.
\end{theorem}

An important step in our investigation is to make the converse of Siegel's theorem explicit. Concretely, assume $g(C)=0$ and write $C_0=C\setminus D$. If $\deg(D)\le 1$ (thus $C_0$ is \emph{log rational}) and $C_0$ contains at least one integral point, then $C_0$ admits an integral parametrization from $\dA^1$. If $\deg(D)=2$, then $C_0$ admits a parametrization by a torus. For this reason we call such curves \emph{toroidal}. Integral points on such curves are dominated by the units $\fo_K^\times$ of the smallest extension $K/k$ over which both points in $D$ are defined. In fact, if $C_0$ contains an integral point, it contains a subgroup of $\fo_K^\times$ of full rank. See Propositions~\ref{prop:a1para}, \ref{prop:torus-para-integral}, and Corollary~\ref{cor:torus-infinite-numberfield} for precise statements.

Based on these parametrizations of suitable subsets of integral points on rational curves, we are able to fully determine their approximation constants (Propositions~\ref{prop:approximation-log-rational} and~\ref{prop:approximation-toroidal}).

\subsection{Structure of the article} We begin by introducing the notion of \emph{log schemes} over a general Dedekind domain and studying the structure of rational curves that potentially possess infinitely many integral points in Section~\ref{sec:schemes-and-curves}.
After defining integral approximation constants in Section~\ref{se:intappcst}, in Section~\ref{se:ratcurve} we compute their values on all types of curves studied in Section~\ref{sec:schemes-and-curves}: \emph{log rational} and \emph{toroidal} ones; we propose a conjecture predicting that the best approximations should be achievable on such rational curves.
In Section~\ref{se:ex}, we study an explicit example in detail and show that the best approximations in a Zariski open subset can be simultaneously found in three different types of rational curves. Finally, in Section~\ref{se:toric} we verify our conjecture for a family of toric varieties. 

\subsection{Notation and conventions}
For a number field $k$, let $\fo_k$ be its ring of integers. Moreover, denote by $\Omega_k$ the set of all places of $k$, and for each place $v\in \Omega_k$, denote by $k_v$ the completion of $k$ at $v$. Let $\infty_k\subset \Omega_k$ be the set of archimedean places. If $v\in \Omega_k\setminus \infty_k$ is a finite place, denote by $\fo_v\subset k_v$ its ring of $v$-adic integers, by $\pi_v\in \fo_v$ a uniformizer, and by $\FF_v = \fo_v/(\pi_v)$ its residue field. For a place $v$, let $|\cdot|_v\colon k_v\to \RR_{\ge 0}$ be the $v$-adic absolute value normalized in the following way: $|p|_p = 1/p$ for primes $p\in \QQ$, $|\cdot|_\infty$ is the usual real absolute value on $\RR= \QQ_\infty$, and if $v\mid w$ for $w\in \Omega_\QQ$, then
\begin{equation*}
	|\cdot|_v = |\cdot|_w^{[k_v:\QQ_w]}.
\end{equation*}
For $\bx\in k_v^n$, let $\norm{\bx}_v$ be the maximum norm.

Write $\RA_k$ for the ring of adèles, and for a finite set $S$ of places, write $\RA^{S}_k = \prod_{v\in \Omega_k\setminus S}k_v \cap \RA_k$ for the ring of adèles off $S$; in particular, write $\RA^\infty_k = \RA^{\infty_k}_k$ for the ring of finite adèles and $\RA^{v}_k$ for the adèles off a single place $v\in \Omega_k$.
Write $k_S = \prod_{v\in S} k_v$, so that $\RA_k=\RA_k^S \times k_S$.

A $k$-variety is a separated scheme of finite type over $k$. For a $k$-variety $X$, the space $X(\RA_k)$ of adelic points is the restricted product of all its local points $X(k_v)$ against $v$-adic integral points $\mX(\fo_v)$, $\mX$ being any $\fo$-model. We also write $X(k_S)=\prod_{v\in S} X(k_v)$ if $S$ is a finite set of places. For any closed point $x\in X$, we write $k(x)$ for the residue field.
For any $k$-scheme $Y$, we write $Y_\red$ for its reduced subscheme.

We shall parametrize integral points over more general bases: Dedekind rings (which should be understood to include fields); throughout Section~\ref{sec:schemes-and-curves}, $\fo$ denotes a Dedekind ring and $k$ its field of fractions. From Section~\ref{se:intappcst} on, $k$ denotes a fixed number field and $\fo=\fo_k$ its ring of integers.

Throughout, we use Vinogradov's notation: for expressions $A$ and $B$ and a variable $x$, the notation $A \ll_x B$ means that there is a constant $C(x)$, which may depend on $x$, such that $A\le C(x)B$. The variant $\gg_x$ is defined symmetrically, and $A\asymp_x B$ means $A\ll_x B$ and $B\ll_x A$. Big-$O$-notation is to be read with the same meaning: $A=O_x(B)$ means $A\ll_x B$. 
We moreover adopt the convention that all ambient varieties, basepoints, and height functions are treated as fixed and all implied constants may depend on them.

\subsection*{Acknowledgments}
This project was initiated when both authors were at Leibniz Universität Hannover. We would like to express our gratitude to Ulrich Derenthal for his support. Further progress was made during a visit of Z. H. to Georg-August-Universität Göttingen. We thank Damaris Schindler for offering hospitality. We are very grateful to the anonymous referee for careful reading of the manuscript and suggestions for improvements. Z. H. acknowledges the organizers of the conference ``Diophantine Days 2025'' held at Westlake University where a preliminary version of this work was reported, and thanks Brian Lehmann, Chunhui Liu and Amos Turchet for their interests and helpful discussions. 

\subsection*{Funding}
Z.~H.\ was supported by National Key R\&D Program of China No.\ 2024YFA1014600. F.~W.\ was supported by Deutsche Forschungsgemeinschaft (DFG) grant 398436923 (GRK2491).

\section{Log schemes and integral points on curves}\label{sec:schemes-and-curves}
Throughout this section, let $\fo$ denote a Dedekind ring and $k=\Frac(\fo)$ its field of fractions.
\subsection{Log schemes and integral points}
Let $\mU$ be a non-proper scheme that is embedded in a projective scheme $\mX$, all over $\fo$.
By blowing up the part of  the complement $\mD = \mX\setminus \mU$  that is of codimension 2 or greater, we may assume that $\mD$ is a divisor, in which case we call $(\mX,\mD)$ a \emph{log pair}. 
 An \emph{integral point} on the pair $(\mX,\mD)$ (or on the open part $\mU$) is a morphism $\Spec \fo \to \mU$. The set of integral points is denoted by $\mU(\fo)$. It follows from the definition that a rational point in $\mX(k)$ is integral if and only if the unique integral point in $\mX(\fo)$ that it extends to does not specialize to a point on the divisor $\mD$ modulo any finite place. We call an integral point \emph{generically regular} if the image of the generic point of $\Spec \fo$ is regular --- that is, if it is regular as a rational point.
\begin{definition}
	A \emph{log scheme} over $\fo$ is a pair $(\mX, \mD)$ consisting of a separated $\fo$-scheme $\mX$ of finite type and a reduced effective divisor $\mD$.
	A log scheme $(\mX,\mD)$ is called
	\begin{enumerate}
		\item \emph{flat} if both $\mX$ and $\mD$ are,
		\item \emph{projective} if $\mX$ (hence $\mD$) is, and 
		\item \emph{generically regular} if the generic fiber $X = \mX_k$ is regular and the generic fiber $D = \mD_k$ of the divisor geometrically has strict normal crossings.
	\end{enumerate}
	As a useful shorthand, a \emph{nice log scheme} is a flat, projective, and generically regular log scheme.

	Moreover, for any log scheme $(\mX, \mD)$, we write $\mX_0 = \mX\setminus \mD$ (whenever the divisor $\mD$ is clear from context).
\end{definition}

Note that $\mD$ being flat means that it is the closure of its generic fiber in $\mX$. Moreover, if $(\mX,\mD)$ is a log scheme over $\fo$, then the generic fibers form a log scheme $(X,D)$ over $k$, with $X\setminus D$ denoted by $X_0$ whenever the divisor is clear from context.

\subsection{Rational curves with infinitely many integral points}\label{se:lograttoridal}

Analogously to McKinnon's rational approximation constant, we expect that the integral approximation constant we are about to define is obtained by points on rational, possibly singular, curves. We shall thus turn our attention to curves in this more general setting for now.
As curves whose genus is at least one can never contribute to approximation constants as their set of integral points is finite by Siegel's theorem, we need only consider rational curves.
On the other hand, any such curve that contains infinitely many integral points could, prima facie, possibly achieve the approximation constant. Our first aim is thus to determine such curves.

Recall that a $k$-rational curve on a projective $k$-variety $X$ is (the image of) a non-constant morphism $\PP^1_k\to X$. If the target carries a log structure, it induces one on the curve.
\begin{definition}
	Let $(\mX,\mD)$ be a log scheme over $\fo$ with generic fiber $(X,D)$ and $\phi\colon \PP^1_k\to X$ a rational curve. In this case, we
	\begin{enumerate}
		\item call $C=\im(\phi)$ a rational curve on $X$, equipped with a log structure $(C,D\cap C)$, and
		\item similarly, if $\mC$ is the Zariski closure of $C$, equip it with the log structure $(\mC,\mD\cap \mC)$ and call it a rational curve on $(\mX,\mD)$.
	\end{enumerate}
	Analogously to before, we write $C_0 = C\setminus D$ and $\mC_0 = \mC \setminus \mD$ whenever $\mD$ is clear from context.
\end{definition}

Let $C$ be a rational curve over $k$ and $C_0 = C\setminus D$ be an open subset.
It follows from Siegel's Theorem (Theorem~\ref{thm:Siegel}) that if $k$ is a number field and $D$ 
has at least three geometric points, then it contains only finitely many integral points; this leaves three cases: $\deg(D_\red)\in \{0,1,2\}$.

We begin with the former two cases: that of log rational curves, close analogs to rational curves on proper varieties.
\begin{definition}
	A \emph{log rational curve} over $\fo$ is a flat, projective log scheme $(\mC,\mD)$ with generic fiber $(C,D)$ such that
	\begin{enumerate}
		\item $\mC$ is birational to $\PP^1_\fo$, and
		\item $\deg(\phi^{-1}(D))\le 1$, where $\phi\colon \PP_k^1\to C$ is the normalization.
	\end{enumerate}
	Note that if $(\mC,\mD)$ is a log rational curve over $\fo$, then its generic fiber $(C,D)$ is a log rational curve over $k$.
\end{definition}

Log rational curves that can contribute to an approximation constant for a boundary point need to contain that point, hence meet the boundary in precisely one point. This type of log rational curve is dubbed \emph{$\dA^1$-curve} by Chen and Zhu~\cite[Def.~1.11]{Chen-Zhu}.

It turns out such a curve contains infinitely many integral points as soon as it contains at least one (if $k$ is a number field, by strong approximation for $\dA^1$ away from the infinite places, this happens as soon as it has integral points everywhere locally). In fact, a subfamily of integral points is parametrized by the affine line in this case.

\begin{proposition}\label{prop:a1para}
	Let $(\mC,\mD)$ be a log rational curve over $\fo$.
	If $\mC_0(\fo)$ contains a generically regular point, then there is a birational morphism $\phi\colon \dA_{\fo}^1\to \mC_0$.
\end{proposition}
\begin{proof}
The idea of proof is that, simply put, starting with a parametrization $\psi\colon \dA_k^1\to C_0$ that maps $0$ to an integral point, the composed morphism with $\dA_k^1 \to \dA_k^1,\ x\mapsto cx$ for a sufficiently divisible $c$ spreads out to integral models.

	Let $x\in \mC_0(\fo)$ be such a generically regular point, and let $\mC \to \PP^n_{\fo}$ be an embedding of the projective scheme $\mC$. 
	Let $\psi\colon \dA_k^1\to C_0$ be a suitable restriction of the normalization of $C$: either to the preimage of $C\setminus D$ (if $C\cap D\ne\emptyset$) or the preimage of $C\setminus \{y\}$ for an arbitrary $y\ne x$ otherwise.
	As $x$ is regular, its preimage is rational, so by composing $\psi$ with a translation, we may assume $\psi(0)=x$.
	The affine cone $\tau\colon \dA_{\fo}^{n+1}\setminus \{0\} \to \PP^n_{\fo}$ is a $\Gm$-torsor, hence so is its fiber over $\psi(0)$.
	Denoting by $c \in H^1_\et(\Spec \fo;\Gm)=\cl_k$ the class of the latter torsor and by $\pi\colon \PP^n_\fo \to \Spec \fo$ the structure morphism, let $\tau_c\colon \mY\to \PP^n_{\fo}$ be a torsor of class $[\tau]-\pi^*c\in H_\et^1(\PP_\fo^n;\Gm)$, whose generic fiber is $\mY_k \cong \dA^{n+1}_k \setminus \{0\}$ (see constructions by Frei and Pieropan~\cite[Def~2.4, Prop~2.5, Thm.~2.7]{FP} for a concrete description of this twist). Its fiber above $\psi(0)$ is now a trivial torsor, hence $\psi(0)$ lifts to an integral point $y\in \mY(\fo)$.
	Describing $\psi$ by means of homogeneous polynomials lifts $\psi$ to a morphism $\varrho\colon \dA^1_k \to \dA_k^{n+1}\setminus \{0\} \cong \mY_k$, and multiplying this lift with a suitable constant in $k^\times$ guarantees that $\varrho(0)=y$.

	The twisted torsor $\mY$ is a quasi-affine scheme of the form $\mY = \Spec A\setminus V(I)$ with $A$ of finite type and $I\subset A$. Write  $A = \fo[a_1,\dots,a_s]$ for $s\ge 1$ and generators $a_1,\dots, a_s\in A$. The images of these generators under $\varrho$ are polynomials
	\begin{equation*}
		\varrho^*(a_i) = \sum_{0\le j\le d} \lambda_{i,j}t^j \in k[t]
	\end{equation*}
	for some $d\ge 0$ and $\lambda_{i,j}\in k$ for $1\le i \le s$, $0\le j\le d$.
	That $\varrho(0)$ is integral means that the composition
	\begin{equation*}
		\begin{tikzcd}
			A \arrow[r, "\varrho^*"] & k[t] \arrow[r, "t\mapsto 0"] & k
		\end{tikzcd}
	\end{equation*}
	factors through $\fo$. Therefore, the constant terms $\lambda_{i,0}$ are integral for all $1\le i \le s$.
	Writing $\lambda_{i,j} = p_{i,j}/q_{i,j}$ with integral $p_{i,j}$ and $q_{i,j}$ for the remaining $1\le i \le s$ and $1\le j\le d$, let $b = \prod_{i,j}q_{i,j}$ be the product of denominators and $t' = b^{-1}x$. Then $\varrho^*$ factors through $\fo[t']$, hence $\varrho$ spreads out to a morphism
	\begin{equation*}
	\varrho\colon \dA_{\fo}^1 \cong \Spec \fo[t'] \to \Spec A
	\end{equation*}
	of integral models.
	
	The preimage $\varrho^{-1}(V(I))$ does not meet the generic fiber, hence is contained in the fibers above a finite set $S\subset\Spec\fo$. Let $a\in \fo$ be divisible by all primes in $S$ and write $\varrho' = \varrho\circ m_a$, where $m_a\colon \dA_{\fo}^1 \to \dA_{\fo}^1$ denotes multiplication by $a$. For every closed point $z$ of $\dA^1_\fo$ over a prime $\fp\in S$, its image under $\varrho'$ is
	$\varrho'(z) = \varrho(az)$, which is congruent to $\varrho(0) = y$ modulo $\fp$, hence $\varrho'(z)\in \mY\subseteq \Spec A$. It follows that $\im\varrho' \subseteq \mY$: if a nonclosed point were to map into the closed set $\Spec A \setminus \mY$, so would a closed one. In fact, $\varrho'$ maps into the closed subscheme $\tau_c^{-1}(\mC)$ (as its generic fiber does), and as $y\in \tau_c^{-1}(\mC_0)$, the same argument as before further implies that $\varrho'$ factors through $\tau_c^{-1}(\mC_0)$. We may now obtain the desired morphism $\phi$ by composing $\varrho'$ with $\tau_c$.
\end{proof}
    
\begin{example}
        Let us exhibit a simple example of reparametrization in the proof above. Let $k$ be a number field, and consider the one-parameter family of cuspidal curves
		\begin{equation*}
			C_a\colon ax^2z=y^3
		\end{equation*}
		in $\PP^2_k$. Each one admits the parametrization $\PP_k^1\to C_a$ by
		\begin{equation*}
			\phi\colon [u:v]\mapsto [au^3:au^2v:v^3].
		\end{equation*}
		Let $D$ be the hyperplane $(x=0)$. Then we see that $C_a\cap D=[0:0:1]$ and all places dividing $a$ obstruct $\phi([1:1])= [1:1:1/a]$ from being integral. However, if we use the alternative parametrization
		\begin{equation*}
			\phi'\colon [u:v]\mapsto [u^3:au^2v:a^2v^3],
		\end{equation*}
		the image of every integral point $[1:v]\in \dA^1(\fo)$ becomes integral. 
    \end{example}

\begin{remark}
	The condition that the existing integral point is regular is necessary and comparable to the necessity of a smooth rational point on a geometrically rational curve to make it rational. For instance, consider the log rational curve
	\begin{equation}\label{eq:nodal-nonexample}
		C\colon 2y^2 = x^2\left(2x+1\right)
	\end{equation}
	in $\dA^2_\ZZ\subseteq\PP^2_\ZZ$. It admits the integral point $(0,0)$, but no other integral points. Indeed, for an integral solution with $y\ne 0$, the $2$-adic valuation of the left-hand side of~\eqref{eq:nodal-nonexample} is odd, but the one on the right cannot be.
\end{remark}

The remaining case is that of a rational curve meeting the boundary in two points. The log genus of such curves is one, and they behave much like elliptic curves; in particular, they are algebraic groups (more precisely, tori), as soon as they possess a rational point.
Two types of embeddings are to be distinguished: those that map the two boundary points to different points and those that identify them (i.e., nodal singular curves). 

\begin{definition}
	A \emph{toroidal curve} over $\fo$ is a flat, projective log scheme $(\mC,\mD)$ with generic fiber $(C,D)$ such that
	\begin{enumerate}
		\item $\mC$ is birational to $\PP^1_\fo$, and
		\item $\deg(\phi^{-1}{D}) = 2$, where $\phi\colon \PP^1\to C$ is the normalization, and the residue field of every point in $\phi^{-1}(D)$ is separable over $k$.
	\end{enumerate}
	Note that if $(\mC,\mD)$ is a toroidal curve over $\fo$, then its generic fiber $(C,D)$ is a toroidal curve over the field of fractions $k$.

	We say that a toroidal curve $(\mC,\mD)$ with generic fiber $(C,D)$ is \emph{nodal} if $D$ consists of a single point, that is, if the normalization $\phi$ identifies the two boundary points on $\PP^1$ into a nodal singularity.
\end{definition}

The following lemma shows that, roughly speaking, a toroidal curve $C$ admits a parametrization by a torus $T = \PP^1\setminus \phi^{-1}(D) \to C\setminus D$.

\begin{lemma}\label{le:torus-para-rational}
 Let $(\mC,\mD)$ be a toroidal curve over $\fo$ with generic fiber $(C,D)$. Let $K=k(x)$ be the residue field at one of the points $x\in\phi^{-1}(D)$, where $\phi\colon \PP^1\to C$ is the normalization. Then $\PP^1\setminus \phi^{-1}(D)$ is isomorphic to a torus $T$, where
 \begin{enumerate}
	\item $T = \Gm$ if $K=k$ and
	\item $T$ is the norm-one torus $R^1_{K/k}\Gm$ if $K\ne k$.
 \end{enumerate}
\end{lemma}
\begin{proof}
A parametrization $\Gm\to C\setminus D$ exists over the residue field $K\coloneq k(x)$ --- either $k$ or a quadratic extension thereof. If $K=k$, we are done.
Now assume that $K/k$ is quadratic, and denote by $T$ the norm-one torus $R^1_{K/k}\Gm$.
Then $\PP^1\setminus \phi^{-1}(D)$ is a $K/k$-form of $\Gm$, that is, a $T$-torsor, by Galois descent.
As $(\PP^1\setminus \phi^{-1}(D))(k)\neq\emptyset$, this torsor is trivial.
\end{proof}

If $K/k$ is a separable extension (as is the case in Lemma~\ref{le:torus-para-rational}), the integral closure $\fO$ of $\fo$ in $K$ is a finite extension of $\fo$, itself a Dedekind ring, and flat as it is torsion free. The restriction of scalars $R_{\fO / \fo}\Gmok$ thus exists~\cite[\S\,7.6~Thm.~4]{MR1045822}, and as the norm restricts to the rings, it induces a homomorphism
\begin{equation*}
	N_{\fO/\fo}\colon R_{\fO / \fo}\GmO \to \Gmok
\end{equation*}
of $\fo$-group schemes by working on local bases, whose kernel
\begin{equation*}
	R^1_{\fO / \fo}\GmO = \ker (N_{\fO/\fo}),
\end{equation*}
we call the \emph{norm-one torus}.

\begin{definition}
	Let $(\mC,\mD)$ be a toroidal curve over $\fo$, and
    keep the notation from Lemma~\ref{le:torus-para-rational}
	\begin{enumerate}
		\item The field $K$ is called the \emph{splitting field} of $(\mC,\mD)$.
		\item The torus $T$ is called \emph{the torus of $(\mC,\mD)$} and admits a birational, surjective morphism $\tau \colon T\to C_0$. 
		\item If $K=k$, then $\mT = \Gmok$ is called the \emph{standard model} of $T$. Otherwise, let $\fO$ be the integral closure of $\fo$ in $K$; we then call the norm-one torus
		\begin{equation*}
			\mT=R^1_{\fO / \fo}\GmO
		\end{equation*}
		the \emph{standard model} of $T$.
	\end{enumerate}
	Note that these definitions are independent of the choice of $x\in \phi^{-1}(D)$.
\end{definition}

\begin{proposition}\label{prop:torus-para-integral}
Let $(\mC,\mD)$ be a toroidal curve over $\fo$, $K$ its splitting field, and $\tau\colon T\to C_0$ its torus, with standard model $\mT$.
\begin{enumerate}
	\item The set $\mC_0(\fo)$ of integral points is contained in the image of finitely many cosets of $\mT(\fo)$ in $T(k)$.
	\item Assume that $\mC_0(\fo)$ contains a generically regular integral point $x$. Let $y$ be its unique preimage in $T(k)$. Then there is a finite set $S\subset \Spec \fo$ and $e\ge 1$ such that $\tau(G_{S,e}y)\subset \mC_0(\fo)$, where
	\begin{equation*}
		G_{S,e} = \ker\left(\mT(\fo) \to \prod_{\fp\in S} \mT(\fo/\fp^e)\right).
	\end{equation*}
\end{enumerate}
\end{proposition}

\begin{proof}
If the set of integral points only consists of singular points, all assertions are trivially true (as there can be only finitely many singular points), so we may assume that $\mC_0$ admits a generically regular integral point $x\in \mC_0(\fo)$, and we may assume that the identity element $1\in T(k)$ maps to $x$ under the morphism $\tau$.

Let $\fO$ be the integral closure of $\fo$ in $K$.
For the first assertion, multiplying the nonconstant, regular, invertible function $t$ and its inverse $t^{-1}$ on $T_K=\Spec K[t,t^{-1}]$ with a suitable constant $c\in \fo$, we obtain the two regular functions $f=ct$ and $g=ct^{-1}$ on $\mC_0$. In particular, both take on integer values on integral points in $\mC_0(\fo)$. Moreover, $f(x)\mid c^2$, so up to units, $f(x)$ takes on finitely many values. It follows that $\mC_0(\fO)$ is contained in finitely many cosets of
$\mT(\fO) =\fO^\times$;
concretely, every point admits the shape $t=\epsilon d/c$, where $\epsilon\in \fO^\times$ and $d$ runs over the finitely many divisors of $c^2$ up to units.
Now the claim follows for $\mC_0(\fo)\subset \mC_0(\fO)$ as well: the preimages under $\tau$ of any two generically regular points in $\mC_0(\fo)$ differ by an element of $T(k)$, hence they share an orbit under $\mT(\fo) = \mT(\fO)\cap \mT(k)$ if they share a $\mT(\fO)$-orbit.

For the second assertion, we may assume that $y=1$ by multiplying with its inverse. The morphism $\tau$ spreads out to $\mT\to \mC_0$ over the complement of a finite set $S\subset \Spec\fo$ of closed points by avoiding finitely many primes in finitely many denominators. In particular, $\tau(\mT(\fo))\subset \tau(\mT(\fo_\fp)) \subset \mC_0(\fo_\fp)$ for each prime ideal $\fp\not\in S$. Moreover, for each prime $\fp\in S$, the map
$\mT(k) = T(k) \to C(k) = \mC(\fo)$ is continuous with respect to the topology induced by the $\fp$-adic valuation. Hence, the preimage $\tau^{-1}\mC_0(\fo_\fp)$ is open, and as it contains $1$, contains a set of the form
\begin{equation*}
	\ker (\mT(\fo_\fp) \to \mT(\fo_\fp/(\fp\fo_\fp)^e))
\end{equation*}
for some $e\ge 1$.
Thus, each element $g$ of the group
\begin{equation*}
	G_{S,e} = \ker\left(\mT(\fo) \to \prod_{\fp\in S} \mT(\fo/\fp^e)\right)
\end{equation*}
for sufficiently large $e$ maps into $\mC_0(\fo) = \bigcap_{\fp\in \Spec\fo \setminus \{(0)\}} \mC_0(\fo_\fp) \subset \mC(k)$.
\end{proof}

\begin{corollary}\label{cor:torus-infinite-numberfield}
	Let $\fo$ be the ring of integers of a number field $k$ and $(\mC,\mD)$ be a toroidal curve with splitting field $K$.
	Then $\# \mC_0(\fo) = \infty$ if and only if
	\begin{enumerate}
    	\item the set $\mC_0(\fo)$ contains a generically regular point,
    	\item the splitting field $K$ has an infinite group of units (i.e., $\#\infty_K > 1$), and
    	\item if $[K:k]=2$, there is a place $v\in \infty_k$ that splits in $K$ (i.e., $\#\infty_K>\#\infty_k$).
	\end{enumerate}
\end{corollary}
\begin{proof}
	In this setting $\fO=\fo_K$ is the ring of integers of $K$. The statement follows from Proposition~\ref{prop:torus-para-integral} on noting that the subgroup $G_{S,e}$ always has finite index. Moreover, if $[K:k]=2$, then the rank of the group $\fo_K^1=\mT(\fo_K)$ of units of norm $1$ is
	\begin{equation}\label{eq:norm-1-rank}
		\rk\fo_K^1=\#\infty_K - \#\infty_k,
	\end{equation}
	which is positive if and only if at least one archimedean place splits.
	Indeed, 
	\begin{equation*}
		\begin{tikzcd}
			1 \arrow[r] & \fo_K^1 \arrow[r] & \fo_K^\times \arrow[r, "N_{K/k}"] & \fo_k^\times \ar[r]& \displaystyle\prod_{\substack{v\in \Omega_k\\w\mid v}} \Gal(K_w/k_v) 
		\end{tikzcd}
	\end{equation*}
	is exact, the morphism on the right stemming from the Hasse norm theorem and local class field theory, and mapping to a torsion group.
\end{proof}

\begin{remark}
	\begin{enumerate}
		\item 	In contrast to log rational curves, here we do need to assume that $\mC_0(\fo)\ne \emptyset$, which is not implied by the existence of integral points everywhere locally, as strong approximation over a number field fails drastically for tori.
		\item That the second condition is necessary can also be seen in terms of an analytic obstruction: if $K$ has only one archimedean place, then the two regular functions $t$ and $t^{-1}$ on $\Gm$ impose an obstruction at infinity, so that the set of integral points on any integral model of $C_0 \times_{\Spec k} \Spec K$ is finite by~\cite[Thm.~2.3.11]{Wilsch-toric}, including those on $\mC_0 \times_{\Spec \fo} \Spec \fo_K$. 
        \item 	If $(\mC,\mD)$ is a flat, projective log scheme with generic fiber $(C,D)$ and normalization $\phi$ such that $\mC$ is birational to $\PP^1_\fo$ and $\phi^{-1}(D)$ consists of a closed point of degree $2$ defined over an inseparable extension, then $\PP^1\setminus \phi^{-1}(D)$ is a nontrivial form of the additive group $\GG_{\mathrm{a}}$.
	\end{enumerate}
\end{remark}

\section{Integral approximation constants}\label{se:intappcst}
From now on, $k$ is a fixed number field, and $\fo$ is its ring of integers.
\subsection{Preliminaries}
In order to define integral approximation constants, we are in need of distances and height functions.

\subsubsection*{Distances}
If $X$ is a projective $k$-variety and $v$ a place of $k$, we endow the set $X(k_v)$ with a distance function $d_v\colon X(k_v)\times X(k_v)\to \RR_{\ge 0}$ as constructed by McKinnon and Roth~\cite[(2.1--2)]{McK-Roth}. Such a distance function depends on the choice of an embedding of $X$ into a projective space, but only up to a bounded constant. The following properties of this distance function found in their work will be useful.

If $K/k$ is a finite extension, $w\mid v$ is a place above $v$, and $x$ and $y$ are local points in $X(k_v)$, then by~\cite[Prop.~2.1]{McK-Roth},
\begin{equation}\label{eq:distance_extension}
	d_w(x,y) = d_v(x,y)^{[K_w:k_v]}
\end{equation}
(as $|\cdot|_w = |\cdot|_v^{[K_w:k_v]}$);
if $\bx\in k_v^n =  \dA^n(k_v)\subset \PP^n(k_v)$ with $\norm{\bx}_v\le c$ for some $c\in \RR_{>0}$, then
\begin{equation}\label{eq:distance_A1}
	\norm{\bx}_v \ll_c d_v(0,x) \ll \norm{\bx}_v
\end{equation}
by a quick calculation based on the formulas~\cite[(2.1--2)]{McK-Roth} or by \cite[Lemma 2.5]{McK-Roth}. 

If $\phi\colon \PP^1\to X$ is a nonconstant morphism with image $C=\im\phi$ and $x\in \PP^1(k_v)$ is such that the branch of $X$ corresponding to $x$ has multiplicity $m$ in $\phi(x)$, then
\begin{equation}\label{eq:distance_singularity}
	d_v(\phi(x),\phi(y)) \asymp_\phi d_v(x, y)^m
\end{equation}
for $d_v(x, y)\ll 1$~\cite[p.~537]{McK-Roth}.

\subsubsection*{Heights}
A line bundle $L$ on a $k$-variety $X$ equipped with an adelic metric $\norm{\cdot}$ induces a height function $H_L\colon X(k)\to \RR_{>0}$, which only depends on $\norm{\cdot}$ up to a bounded constant. Concretely, for $d\in \ZZ$ and $n\ge 0$, an $\cO(d)$-height on $\PP^n$ is given by 
\begin{equation*}
	H_{\cO(d)}([x_0:\dots:x_n]) = \prod_{v\in\Omega_k} \max\{|x_0|_v,\dots, |x_n|_v\}^d.
\end{equation*}
If $K/k$ is a finite extension, then the base change $L_K$ of $L$ to $X_K$ is equipped with an induced metric, corresponding to a height function $H_{L_K}$ with
\begin{equation}\label{eq:height_extension}
	H_{L_K}(x) = H_L(x)^{[K:k]}
\end{equation}
for all $x\in X(k)$ (again as a consequence of the normalization of the absolute values). See work by Peyre~\cite[\S\,2]{Peyre}, for instance, for more details on these constructions.

\subsection{Approximation over a single place}
Motivated by the work of McKinnon and \cite{McK,McK-Roth}, we propose a variant of the approximation constant. It measures how well a local point $x$ can be approximated by global points $y$ in some set $W$ by attempting to minimize one of the distance functions $d_v(x,y)$ compared to a height $H(y)$.
Mostly, this set $W$ will be the set $\mU(\fo)$ of integral points of a nice log scheme $(\mX,\mD)$, and the point $x$ will normally be a rational point on $D$, but it is convenient to keep the definition more general.
\begin{definition}\label{def:alpha}
	Let $X$ be a projective $k$-variety equipped with a height function $H\colon X(k)\to\RR_{>0}$ associated with a line bundle $L$; let $W\subseteq X(k)$ be a subset. Let $v$ be a place of $k$, $x\in X(k_v)$ be a local point and $d_v(\cdot,\cdot)$ be a distance function. The \emph{approximation constant} $\alpha_v(x, W; L)$ of $x$  with respect to $W$ and $L$ is defined as follows.
	\begin{enumerate}
		\item If $x$ does not lie in the closure of $W \setminus \{x\}$ in $X(k_v)$, then
		\begin{equation*}
			\alpha_v(x, W; L)=\infty.
		\end{equation*}
		\item Otherwise, 
		\begin{equation*}
			\alpha_v(x, W; L) = \inf\left\{
			\alpha \in \RR_{>0}
			\ \middle|\  \substack{
				\text{there is a sequence $(y_n)_{n\in \NN} \subseteq W\setminus \{x\}$ such that}\\
				y_n\to x \text{ and } d_v(x, y_n)^\alpha H(y_n) \text{ is bounded from above}
			}
			\right\}.
		\end{equation*}
	\end{enumerate}
\end{definition}
\begin{remark}
	\begin{enumerate}[label=(\arabic*)]
		\item Similarly to \cite[Rem.~(a) after Def.~2.9]{McK-Roth}, this definition does not depend on the choice of a particular distance $d_v$ or height function $H$ associated with $L$, as they only differ by a bounded factor.
		\item If $\mU$ is an $\fo$-scheme, we write $\alpha_v(x, \mU; L)$ for $\alpha_v(x,\mU(\fo);L)$; similarly if $V\subset X$ is a subvariety, $\alpha_v(x,V;L)$ should be read as $\alpha_v(x,V(k);L)$; in particular, $\alpha_v(x,X;L)$ is the approximation constant studied by McKinnon--Roth \cite[Definition 2.9]{McK-Roth}.
		\item When $L$ is positive enough, the height $H_L(y)$ measures the complexity of a point $y$, and thus the approximation constant measures how well a point $x$ can be approximated by integral points $y\ne x$ of bounded complexity.
		Concretely, the following positivity properties affect the integral approximation constant in the same way that they affect the rational one.
		It is positive if $L$ has the Northcott property ``around $x$''. This happens when $L$ is ample \cite[Prop.~2.14~(d)]{McK-Roth} or more generally, when $L$ is big and $x$ lies outside the augmented base locus (this follows from e.g.~\cite[pf.~of~Prop.~2.12]{Peyre}). Furthermore, it is nonnegative if $H_L$ is bounded from below near $x$; this happens if $L$ is numerically effective and $x$ lies outside the asymptotic base locus \cite[Rem.~(c) after Def.~2.9]{McK-Roth}, for instance if $L$ is nef. These are the cases that we shall mostly be interested in. 
	\end{enumerate}
\end{remark}

Before continuing with our analysis, we discuss some necessary conditions for the constant $\alpha_v(x, W; L)$ to be finite: in particular, a point $x$ needs to lie in the $v$-adic closure of $W$. We shall emphasize the case of archimedean places.

\subsubsection*{Archimedean approximation and Clemens complexes} 
In the absence of obstructions --- specifically, the analytic and Brauer--Manin obstructions [\citenum{Santens}, Thm.~3.12, Def.~3.14, \S\,5; \citenum{Wilsch-toric}, \S\,2.3] --- integral points of bounded log anticanonical height on log Fano schemes tend to equidistribute with respect to an archimedean place $v$ towards a limit measure supported on a subset of the boundary, $D(k_v)$ encoded in the \emph{Clemens complex}.

Let $D$ be a divisor with geometrically strict normal crossings, and let $\cA$ be its set of geometric components. The \emph{geometric Clemens complex} is the set
\begin{equation*}
	\cC(D) = \left\{(A, Z)\ \middle|\ A\subseteq \cA,\ \text{$Z$ is an irreducible component of $\bigcap_{E\in A} E$} \right\},
\end{equation*}
equipped with the partial order defined via $(A,Z)< (A',Z')$ if $A\subsetneq A'$ and $Z\supsetneq Z'$.
In other words, it consists of a vertex $(E,E)$ for every geometric component $E$ of $D$. Simplices are then glued onto the vertices corresponding to components with nonempty intersection.

Let $\Gamma_v$ be the Galois group of $k_v$; it acts on the set $\cA$ of geometric components of $D$.
 The \emph{$v$-analytic Clemens complex}, which encodes intersections in $X(k_v)$, is defined as
\begin{equation*}
	\cC_v(D) = \left\{(A, Z)\ \middle|\ \substack{A\subseteq \cA \text{ is $\Gamma_v$-invariant}\\ \text{$Z$ is an irreducible component of $\bigcap_{E\in A} E$ with $Z(k_v)\ne \emptyset$ }} \right\}.
\end{equation*}
Note that for $(A,Z)\in \cC_v(D)$, the subvariety $Z$ is defined over $k_v$, in general.

Equidistribution statements for integral points of bounded log anticanonical height such as those obtained by Chambert-Loir, Tschinkel~\cite{clt-vector}, and Santens~\cite{Santens} now involve a \emph{residue measure}~\cite[\S\,2.1]{C-L-Tsch} supported on the \emph{minimal strata}
\[
	Z_{D,v} = \bigcup_{(A,Z)\in \cC_v(D) \text{ maximal}} Z
\]
as their limit.
Faces $(A, Z)$ that are maximal, but not maximal dimensional only contribute to lower order terms of the counting function, hence not to the overall limit measure --- restricting the counting problem to points near such a stratum still tends to result in equidistribution on $Z$, though (cf. \cite[p.~3]{Santens}); as only the local distribution is relevant for approximation constants, it is natural to include such maximal, but not necessarily maximal-dimensional, faces in the analysis.

In light of these phenomena, we shall mostly restrict ourselves to studying approximation constants of rational points $x\in Z_{D,v}(k_v)\cap X(k)$, but we remain more general wherever such a restriction is not helpful.

\subsection{Simultaneous approximation}
Properties such as weak or strong approximation involve more than a single place at once, and indeed, approximation constants measuring how well points can be simultaneously approximated over several places are straightforward to define. For a finite set $S$ of places, let $d_S = \sum_{v\in S} d_v$ on
$X(k_S)\times X(k_S)$. 

\begin{definition}
	Let $X$ be a projective $k$-variety equipped with a height function $H$ associated with an ample line bundle $L$, and $S$ a finite set of places.
	Let $\bx = (x_v)_{v\in S}\in X(k_S)$ be a tuple of local points over places in $S$ and $W\subset X(k)$. We write $\Delta_S$ for the diagonal image $X(k)\to X(k_S)$. The \emph{simultaneous approximation constant} $\alpha_S(\bx, W; L)$ of $\bx$  with respect to $W$ and $L$ is defined as follows.
	\begin{enumerate}
		\item If there is no sequence $(y_i)_{i\in \NN}\subseteq W$ of points with $\Delta_S(y_i)\neq\bx$ and $d_S(\Delta_S(y_i), \bx)\to 0$, let
		\begin{equation*}
			\alpha_S(\bx, W; L) =\infty.
		\end{equation*}
		\item Otherwise, let
		\begin{equation*}
		\alpha_S(\bx, W; L)  = \inf\left\{
		\delta\in \RR_{>0}
		\ \middle|\  \substack{
			d_S(\Delta_S(y),\bx)^\delta H(y) \text{ is} \\
			\text{bounded from above for all } y\in W
		}
		\right\}.
		\end{equation*}
	\end{enumerate}
\end{definition}
\begin{remark}
	\begin{enumerate}[label = (\arabic*)]
		\item Analogously to the case of a single place, 
        we omit $\mX\setminus \mD$ from the notation if a nice log pair $(\mX, \mD)$ is clear from context.
		\item The simultaneous approximation constant $\alpha_S(\bx,W;L)$ is bounded from below by $\max_{v\in S}\alpha_v(x_v,W;L)$.
	\end{enumerate}
	
\end{remark}

In fact, the Manin--Peyre adelic equidistribution principle \cite[pp.~230--234]{Peyre} for a Fano variety $X$ implies simultaneous equidistribution of $X(k)$ over any finite set of places whenever it holds.
In the integral setting, such an equidistribution statement would hold within (a subset of)
\begin{equation}\label{eq:X}
	\bX = \prod_{v\in \infty_k} Z_{D,v}(k_v) \times \prod_{v\not\in \infty_k} \mU(\fo_v) \subset X(\RA_k),
\end{equation} cf. [\citenum{C-L-Tsch}, \S\,2.4; \citenum{Santens}, \S\,6; \citenum{Wilsch-toric}, \S\,2.1.3].

Note that the intersection $X(k)\cap \bX$ is always empty (if $D\ne \emptyset$), as the archimedean and nonarchimedean places induce incompatible conditions. If $\bx$ is the diagonal image of a rational point, the integral analog to simultaneous approximation can thus either mean to approximate $x\in \mU(\fo)$ over a finite set of finite places or to approximate a point $x\in D(k)$ that lies on $Z_{D,v}$ for all $v$ over a set $S$ of archimedean places.

\begin{remark}
  One can analogously define an approximation constant for $S$-integral points, where $S$ is a finite set of places containing $\infty_k$. In such a setting, places in $S$ should be treated like archimedean places in the discussion above.
\end{remark}

\subsection{The scope of this article}
Our main focus is to understand approximation constants over a single archimedean place --- over such a place, most of the phenomena that set integral points apart from rational ones are visible.

The behavior of approximation constants for nonarchimedean places differs yet again in several ways. Let $v$ be a finite place. The set $\mU(\fo)$ (if nonempty) is contained in the open-closed subset $\mU(\fo_v)$ of $X(k_v)$, so that the approximation constant is infinite for points in its complement. If strong approximation holds for $\mU$ away from a set of places not containing $v$, then $\mU(\fo)$ is dense in $\mU(\fo_v)$, so that every local integral point $x\in \mU(\fo_v)$ lies in the $v$-adic closure of $\mU(\fo)\setminus \{x\}$. More generally, the integral version of Manin's conjecture due to Santens~\cite{Santens} predicts equidistribution (hence density) in a closed subset of $\bX$ \eqref{eq:X} explicitly described in terms of two obstructions. 
In particular, if $x\in \mU(\fo)$ is a global integral point, it would then automatically lie in the closure of $\mU(\fo)\setminus \{x\}$ in $X(\RA_k^{\infty})$.
In any case, a natural question seems to be how well an integral point in $\mU(\fo)$ can be approximated by other integral points with respect to nonarchimedean absolute values.
We plan to address this problem more thoroughly in future investigations and content ourselves with this very brief discussion for now.

\section{Approximation constants of rational curves}\label{se:ratcurve}
The aim of this section is to compute integral approximation constants of log rational and toroidal curves, using the parametrizations obtained in Section~\ref{sec:schemes-and-curves}.
As in the work of McKinnon and Roth~\cite[Thm.~2.16]{McK-Roth}, they will involve the integer 
\begin{equation*}
	r_{y,v} = \begin{cases}
	0 & \text{if $k(y)\nsubseteq k_v$,} \\
	1 & \text{if $k(y) = k$,} \\
	2 & \text{if $k\subsetneq k(y) \subseteq k_v$,}
	\end{cases}
\end{equation*}
defined for a closed point $y$ on a $k$-variety $X$ and a place $v$ of $k$. This integer will appear in denominators; if it is zero, the resulting quotient should be interpreted as $\infty$.
 
\begin{proposition}[Approximation constants for log rational curves]\label{prop:approximation-log-rational}
    Let $\mC$ be a log rational curve on a nice log scheme $(\mX, \mD)$. Assume that $\mC_0(\fo)$ contains a generically regular point. Let $L$ be a line bundle such that $\deg_L(C)> 0$, $v$ be an archimedean place of $k$, and $x\in (C\cap D)(k)$. Then
	\begin{equation*}
		\alpha_v(x, \mC_0; L) = \frac{\deg_{L}(C)}{\mult_x(C)},
	\end{equation*}
    where $\mult_x(C)$ denotes the multiplicity of $C$ at the point $x$.
\end{proposition}
\begin{proof}
	We write $m_x=\mult_x(C),d = \deg_L(C)$ for simplicity. We begin by observing $\alpha_v(x,\mC_0;L)\ge \alpha_v(x,C;L)$; as the latter constant is the rational approximation constant on $C$, which has been computed by work of McKinnon and Roth~\cite[Thm.~2.16]{McK-Roth} and equals $\frac{d}{m_x}$ (note that $C$ is automatically unibranched as a consequence of its log rationality), this provides the value as a lower bound.

    To get an upper bound, let $\phi\colon\dA^1_{\fo}\to\mC_0$ be a birational morphism as in Proposition~\ref{prop:a1para}, and denote the extension $\PP^1_k\to C$ on generic fibers again by $\phi$ (with $\phi(\infty)=x$).  
    Let $(a_n)_n$ be a sequence of integers in $\fo$ with $|a_n|_v\to \infty$ but $|a_n|_{v'} \ll 1$ for all $v'\in \infty_k\setminus \{v\}$ (the latter condition defines a tube in $k\otimes_\QQ \RR$, hence contains infinitely many elements of the integer lattice).  Then, appealing to~\eqref{eq:distance_singularity}, we can estimate
	\begin{equation*}
		d_v(x, \phi(a_n)) \asymp d_v(\infty,a_n)^{m_x} \asymp |a_n|_v^{-m_x},
	\end{equation*}
	while
	\begin{equation*}
		H_L(\phi(a_n)) = H_{\phi^*L}(a_n) \asymp \prod_{v'} \max\{1,|a_n|_{v'}^d\} \ll |a_n|_{v}^{d},
	\end{equation*}
	so that
	\begin{equation*}
		d_v(x,\phi(a_n))^{d/m_x} H_L(\phi(a_n))
	\end{equation*}
	is bounded, which shows that $\frac{d}{m_x}$ is also an upper bound.
\end{proof}

	\begin{remark}
        In the setting of Proposition~\ref{prop:approximation-log-rational}, it is easy to deduce the same result for the simultaneous approximation constant $$\alpha_{\infty_k}(\Delta_{\infty_k}(x),\mC_0;L)=\frac{\deg_{L}(C)}{\mult_x(C)}.$$
	\end{remark}

The computation of approximation constants on toroidal curves will use the following observations on the geometry of numbers.
\begin{lemma}\label{lem:geom-numbers}
	Let $n\ge 1$ and $\Lambda$ be a lattice in $\RR^n$ of full rank. Then for every $\delta_1,\delta_2>0$, there is a lattice point $x\in \Lambda \cap [\delta_1^{-1},\infty]\times [-\delta_2,\delta_2]^{n-1}$.
\end{lemma}
\begin{proof}
Let $\pi\colon \RR^n\to \RR^{n-1}$ be the projection onto the last coordinates. If $\ker \pi|_\Lambda \ne 0$, then a sufficiently large element of this kernel has the desired property. If not, then the rank of $\pi(\Lambda)$ is larger than $n-1$, hence $\pi(\Lambda)$ is not discrete. Let $y\in \RR^{n-1}$ be an accumulation point of $\pi(\Lambda)$ and $y'\in \pi(\Lambda)$ with $|y-y'|<\delta_2$. Then $y-y'$ is an accumulation point as well, and there are infinitely many points in $\pi(\Lambda)\cap [-\delta_2,\delta_2]^{n-1}$, hence so are there in $\Lambda \cap \RR \times [-\delta_2,\delta_2]^{n-1}$. As $\Lambda$ is discrete and $K=[-\delta_1^{-1},\delta_1^{-1}]\times [-\delta_2,\delta_2]^{n-1}$ is compact, only finitely many of those can lie in $K$. Therefore, we can find a point $x$ outside $K$, which is as claimed after multiplying with a suitable sign.
\end{proof}

\begin{lemma}\label{lem:units-of-norm-1}
	Let $K/k$ be a quadratic extension and $U$ be a subgroup of full rank of the group $\fo_K^1$ of units of norm $1$. Let $v_0$ be a place of $k$ such that $K\subset k_v$; let $w_0$ be a place above $v_0$. For every $\delta>0$, there is a unit $\epsilon\in U$ such that $|\epsilon|_{w_0}<\delta$ and $1/(1+ \delta) < |\epsilon|_{w} < 1+\delta$ for archimedean places $w$ with $w\nmid v_0$.
\end{lemma}
\begin{proof}
	Let 
	\begin{align*}
		\phi \colon \fo_K^\times &\to \RR^{\infty_K},\\\epsilon & \mapsto (\log |\epsilon|_w)_{w\in \infty_K};
	\end{align*}
	we shall begin by describing the image of $\fo_K^1$ under this map by means of a linear subspace.
	The set of archimedean places $\infty_k$ is the union of splitting places $\infty_k^{\mathrm s}$ and inert places $\infty_k^{\mathrm i}$. 

	Let $v\in \infty_k^{\mathrm s}$ be one of the former for now, that is, a place that has two different places $w_1$ and $w_2$ lying above it; note that this case includes $v=v_0$ by assumption, and in this case, label $w_2 = w_0$.
	The two places $w_1$ and $w_2$ are swapped by the nontrivial element $\sigma$ of the Galois group $\Gal(K/k)$, hence
	\begin{equation*}
		1 = |N(\epsilon)|_{w_1} = |\epsilon|_{w_1}|\sigma(\epsilon)|_{w_1} = |\epsilon|_{w_1}|\epsilon|_{w_2}.
	\end{equation*}
	Logarithmically, this translates to
	$f_{v}(\phi(\epsilon))=0$, where
	\begin{equation*}
		f_{v} = x_{w_1}+x_{w_2}\colon \RR^{\infty_K}\to \RR
	\end{equation*}
	is the sum of the two projections.
	Let $b_{v}\in\RR^{\infty_K}$ be the vector with $1$ in the $w_1$-entry and $-1$ in the $w_2$-entry.

	Let now $v\in \infty_k^{\mathrm i}$ be an inert archimedean place and $w$ be the only place above $v$. Then $\RR\cong k_{v} \subset K_{w}\cong \CC$, and $\sigma$ acts by complex conjugation, so that
	\begin{equation*}
		1 = |N(\epsilon)|_{w} = |\epsilon\overline{\epsilon}|_w.
	\end{equation*}
	This implies $f_{v}(\phi(\epsilon))=0$, where $f_{v} = x_w\colon \RR^{\infty_K}\to \RR$ is the projection.
	
    Let $H\subset \RR^{\infty_K}$ be the linear subspace cut out by the $\# \infty_k$ linearly independent linear forms $(f_{v})_{v\in\infty_k}$. It is thus of dimension $\#\infty_K - \#\infty_k$ with basis $(b_{v})_{v\in \infty_k^{\mathrm s}}$.
    Note that $H$ contains the sublattice $\phi(\fo_K^1)$ of $\phi(\fo_K^\times)$, which is of full rank by~\eqref{eq:norm-1-rank}.
	In the basis $(b_{v})_{v\in\infty_k^{\mathrm s}}$, ordered so that $v_0\in \infty_k^{\mathrm s}$ comes first, the desired properties for the absolute values translate to
	\begin{equation*}
		\phi(\epsilon) \in [- \log \delta ,\infty] \times [-\log (1+\delta), \log(1+\delta)]^{\#\infty_k^s -1}.
	\end{equation*}
	Now the claim follows from Lemma~\ref{lem:geom-numbers}.
\end{proof}
    
\begin{proposition}[Approximation constants for toroidal curves]\label{prop:approximation-toroidal}
Let $\mC$ be a toroidal curve on a nice log scheme  $(\mX,\mD)$ such that $\#\mC_0(\fo) = \infty$.
Let $x\in (C\cap D)(k)$,  $v\in\infty_k$, and $L$ be a line bundle on $X$ with $\deg_L(C) > 0$.
\begin{enumerate}
 \item If $C$ is not nodal, then
	\begin{equation*}
		\alpha_v(x,\mC_0;L) = \frac{\deg_{L}(C)}{\mult_{x}(C)}.
	\end{equation*}
    \item If $C$ is nodal, let $C_1$ and $C_2$ be its two branches corresponding to $y_1$ and $y_2$ respectively. Then
	\begin{equation*}
	\alpha_v(x,\mC_0;L) =
        \frac{\deg_{L}(C)}{\max_{i=1,2}\{r_{y_i,v}\mult_{x}(C_i)\}}.
	\end{equation*}
\end{enumerate}  
\end{proposition}
\begin{proof}
	As $\alpha_v(x,\mC_0;L) \geq\alpha_v(x,C;L)$, all expressions on the right hold as a lower bound by~\cite[Thm.~2.16]{McK-Roth}, so that we are only left to construct sufficiently good approximating sequences to provide upper bounds. 
 	Let $K$ be the splitting field of the toroidal curve $C$, $\phi\colon \PP^1\to C$ be its normalization, and write $\{y_1,y_2\}=\phi^{-1}(D)\subset \PP^1(K)$ for the (Galois invariant) preimage of the boundary. Let $\tau\colon T=\PP^1\setminus \{y_1,y_2\}\to C_0$ be the torus of $C$. 

	We start with the non-nodal case, in which case $K=k$.
    By applying a linear change of variables and appealing to Proposition~\ref{prop:torus-para-integral}, we may assume that $y_1=0$, $y_2=\infty$, $\phi(y_1)=x$, and that a subgroup $G$ of $\Gm(\fo)=\fo^\times$ of full rank is mapped into $\mC_0(\fo)$ by $\tau$. Let $\epsilon$ be an element of this subgroup of infinite order with $|\epsilon|_v<1$ and $|\epsilon|_{v'}>1$ for all $v'\in \infty_k\setminus\{v\}$ (which exists by Dirichlet's unit theorem), and consider the sequence $(\phi(\epsilon^k))_{k\in\NN}$. An argument analogous to that in the proof of 
	Proposition~\ref{prop:approximation-log-rational} now yields the claimed constant.

	Assume from now on that $C$ is nodal. Note that both points $y_1,y_2$ have the same residue field $K$. Write for simplicity $d= \deg_L(C)$ and  $m_{y_i}=\mult_{x}(C_i)$.  If $K=k$ (and in particular, $r_{{y_1},v}=r_{{y_2},v}=1$), with the same convention $y_1=0$, $y_2=\infty$ and the same choice of $\epsilon$ as in the non-nodal case result in the distances
	\begin{equation*}
		d_v(x,\tau(\epsilon^n)) \ll d_v(0,\epsilon^{n})^{m_{y_1}} = |\epsilon|_v^{n m_{y_1}}
	\end{equation*}
	and 
	\begin{equation*}
		d_v(x,\tau(\epsilon^{-n})) \ll d_v(\infty,\epsilon^{-n})^{-m_{y_2}} = d_v( 0,\epsilon^n)^{m_{y_2}} = |\epsilon|_v^{n m_{y_2}},
	\end{equation*}
	while its height is $$H_L(\epsilon^n) \asymp \prod_{v'}\max\{1,|\epsilon^n|_{v'}\}^d = \prod_{v'\ne v} |\epsilon^n|_{v'}^d = |\epsilon|_v^{-nd}$$ by the product formula,  and $H_L(\epsilon^{-n}) \asymp |\epsilon|^{-dn}_v$ by definition, achieving the claimed constant.

	We are thus only left with the case $k\subsetneq K$ (with $[K:k]=2$).
	For now, assume that $K\nsubseteq k_v$, so that the compositum $Kk_v/k_v$ is a quadratic extension of local fields, and $r_{y_i,v}=0$. Then $y_1$ and $y_2$ lie in the open set $\PP^1(Kk_v)\setminus \PP^1(k_v)$, and each of them contains a sufficiently small compact neighborhood that does not meet $\PP^1(k_v)$ (as else $\phi^{-1}(x)$ would meet $\PP^1(k_v)$). 
    This implies that there is a neighborhood $V$ of $x$ in $C(k_v)$ with $V\cap \mC_0(\fo)=\emptyset$, where  $d_v(x, \cdot)$ is bounded from below, and hence the approximation constant is infinite.

	We may thus finally assume $K\subseteq k_v$, in which case $r_{{y_1},v}=r_{{y_2},v}=2$.  There is a $d\in k^\times$ such that $K=k(\sqrt{d})$ and $\sqrt{d}\in k_v$, that is, $K_w = k_v$ for $w\mid v$. Let $w_1$ and $w_2$ be the two places above $v$.
	By Proposition~\ref{prop:torus-para-integral}, $T$ is a norm-one torus and that a subgroup $G$ of $\mT(\fo)$ of full rank is mapped into $\mC_0(\fo)$ under $\tau$. Let $\lambda$ be a linear change of variables of $\PP^1_K$ mapping $0$ to $y_1$, $\infty$ to $y_2$ and $\mT(\fo)$ to the norm-1-units $\fo_K^1\subset\Gm(\fo_K)\subset \PP^1(K)$.
	Composing with $\phi$ we get a morphism $\psi = \phi\circ\lambda\colon \PP^1_K\to C$ with $\psi(0)=\psi(\infty)=x$. For every $\delta>0$, let $\epsilon_\delta\in G$ be a unit with $|\epsilon_\delta|_{w_1}<\delta$ and $1/(1+\delta)< |\epsilon_\delta|_{w'}<1+\delta$ for all $w'\in \infty_K\setminus \{w_1,w_2\}$ as constructed in Lemma~\ref{lem:units-of-norm-1}. Computing $d_w(0,\epsilon_\delta^n)$ as before and appealing to~\eqref{eq:distance_extension} with $[K_w:k_v]=1$, we again get
	\begin{equation*}
		\begin{aligned}
			d_v(x, \tau(\lambda^{-1}(\epsilon_\delta))) &= d_{w_1}(x, \tau(\lambda^{-1}(\epsilon_\delta))) \asymp d_{w_1}(\psi(0),\psi(\epsilon_\delta)) \\
			&\ll  |\epsilon_\delta|_v^{m_{y_1}}
			\ll \delta^{-m_{y_1}}.
		\end{aligned}
	\end{equation*}
	Computing the $L_K$-height yields
	\begin{equation*}
		\begin{aligned}
			H_{L_K}(\tau(\lambda^{-1}(\epsilon_\delta^n))) & = \prod_{w'\in \infty_K} \max\{1,|\epsilon_\delta|_w\}^d \ll |\epsilon_\delta|_{w_2}^d (1+O(\delta)) \\
			& = |\epsilon_\delta|_{w_1}^{-d} (1+O(\delta)),
		\end{aligned}
	\end{equation*}
	whence
	\begin{equation*}
		H_{L}(\tau(\lambda^{-1}(\epsilon_\delta))) = H_{L_K}(\tau(\lambda^{-1}(\epsilon_\delta)))^{1/2} = |\epsilon_\delta|_{w_1}^{-d/2} (1+O(\delta)),
	\end{equation*}
	achieving the claimed approximation constant and finishing the proof.
\end{proof}

    \begin{remark}
		In stark contrast to log rational curves, the simultaneous approximation constant $\alpha_{\infty_k}(\Delta_{\infty_k}(x),\mC_0;L)$ can differ dramatically from the one with respect to a fixed place, depending on whether the toroidal curve $C$ is nodal or not. Keep all notation from Proposition~\ref{prop:approximation-toroidal} and its proof, and let $t\in \fo_K[\mT]\cong \fo_K[\Gm]= \fo_K[t,t^{-1}]$ be the nonconstant invertible regular function on $\mT$ vanishing in $y_1=0$ (hence measuring the distance to $y_1$). 	 
        \begin{enumerate}
            \item On nodal curves, note that the distances for places $w$ of $K$ behave like
		\begin{equation*}
			d_w(\tau(y),x) \asymp \min\{|f(y)|_w, |f(y)|^{-1}_w\}
		\end{equation*} for $y\in T$,
		as $x$ can be approximated by approximating either $0$ or $\infty$ on $\PP_K^1\supset T_K$, and an argument analogous to that in Proposition~\ref{prop:approximation-toroidal} shows that $\alpha_{\infty_k}(\Delta_{\infty_k}(x),\mC_0;L)$ is finite. 
        \item The behavior on non-nodal curves changes, though.	 In this case, the local distance functions to $x=\tau(y_1)\ne \tau(y_2)$ for places $w$ of $K$ behave like
        \begin{equation*}
			d_w(\tau(y),x) \asymp \min\{|f(y)|_w, 1\}
		\end{equation*}
		for $y\in \mT(k)$. Moreover, for $z\in T(k)$ and $y=zy'\in z\mT(\fo)$ in the corresponding coset,
		\begin{equation*}
			\prod_{w\in \infty_K} |f(zy')|_w = \prod_{w\notin\infty_K} |f(z)f(y')|_w^{-1} \gg_z 1,
		\end{equation*}
		the total distance is bounded from below on every coset:
		\begin{equation*}
			d_{\infty_k}(\tau(y),x) \asymp \sum_{v\in \infty_K} \prod_{w\mid v} |f(zy)|_w \gg_z 1.
		\end{equation*}
		As finitely many cosets sweep $\mC_0(\fo)$, it follows that the constant $\alpha_{\infty_k}(\Delta_{\infty_k}(x),\mC_0;L)$ is infinite for non-nodal toroidal curves. This in particular holds true for $\alpha_\infty$, where $\infty$ is the unique archimedean place of the rationals or an imaginary quadratic field --- but note that in this case, the number of integral points on a non-nodal toroidal curve is already finite by Corollary~\ref{cor:torus-infinite-numberfield}.
        \end{enumerate}	
    \end{remark}

\subsection{A conjecture}
Fano varieties contain an abundance of rational curves over an
algebraic closure, being geometrically rationally
connected~\cite{MR1191735,MR1158625}. For this geometric reason, it is expected that rational points tend to be abundant on Fano varieties: a qualitative conjecture
of Colliot-Thélène \cite[Conj.~14.1.2]{C-TSKbook} implies weak weak approximation (that is, weak approximation away from some finite set of places) and in
particular, that the set of rational points is not thin as soon as it is nonempty.
Moreover, Manin's conjecture~\cite{MR974910} was proposed, measuring quantitatively how rational points of bounded anticanonical height grow.

Weak Fano varieties, whose anticanonical bundle is big and nef, but not necessarily ample,
are a slight generalization, and this property of the anticanonical bundle appears
in variants of Manin's problem, e.g.~\cite{MR4472281} --- the associated heights almost satisfy the Northcott property, except that they might stay bounded
on an infinite set of points on strict subvarieties.
The logarithmic analog to these varieties thus seems like a natural candidate to
expect approximation constants to be governed by rational curves.

\begin{definition}
	A \emph{weakly log Fano scheme} over $\fo$ is a nice log scheme $(\mX, \mD)$  such that the log anticanonical class $-K_X(D) = -K_X-D$ of its generic fiber $(X, D)$ is big and nef.
\end{definition}

 It follows from \cite[Cor.~5.4]{KeelMcKernan} that whenever $-K_X(D)$ is ample, such an $X$ is covered by a family of log rational or toroidal curves. Based on our previous analysis, we now propose the following analog of
 McKinnon's conjecture for integral points.
 We shall restrict ourselves to schemes that indeed admit an abundance of integral points:
 integral points on log Fano varieties might still be thin --- but note that
 a conjecture of Santens~\cite{Santens} expects this to be fully explained by concrete obstructions.
 For the sake of simplicity, we shall directly assume that this is not the case, that is, that the \emph{integral Hilbert property} holds:
 recall that a subset $W\subset X(k)$ is \emph{thin} if there exists a generically finite rational map $f\colon Y\to X$ from a $k$-variety $Y$ that does not admit a section such that $W\subset f(Y(k))$, and that $\mU = \mX\setminus \mD$ satisfies the integral Hilbert property if its set $\mU(\fo)$ of integral points is not thin. This latter condition is stronger than the set of integral points being Zariski dense. Moreover, the equidistribution properties discussed in previous sections imply the integral Hilbert property.

    \begin{conjecture}\label{conj:integralmckinnon}
       Let $(\mX,\mD)$ be a weakly log Fano scheme satisfying the integral Hilbert property, and let $x\in D(k)$ lie on a minimal stratum of $D = \mD_k$. Let $v$ be an archimedean place and $L$ be an ample line bundle. Then there exists a rational curve $\mC$ on $(\mX,\mD)$ which is either log rational or toroidal such that $$\alpha_v(x, \mU;L)=\alpha_v(x, \mC_0;L).$$ 
    \end{conjecture}

	Nevertheless, it remains interesting to study approximation constants on schemes failing
	these assumptions. In fact, it seems likely that rational curves still govern approximation 
	constants on a larger class of varieties --- similarly to how Manin's conjecture
	holds for many \emph{almost Fano} varieties as defined by Peyre~\cite[Hyp.~3.27]{Peyre} such as
	all toric varieties, even though there also exist counterexamples in this more general setting~\cite[\S\,5.1]{MR4472281}.

    \begin{remark}
        \begin{enumerate}
        \item Curves achieving $\alpha_v(x, \mU;L)$ need not be unique.
        \item By Proposition~\ref{prop:approximation-toroidal}, nodal toroidal curves are more likely to contribute to a smaller approximation constant because of the factor $r_{y,v}$.
        \item Proposition~\ref{prop:a1para} implies in particular that the number of integral points of $L$-height at most $B$ on a log rational curve grows like a power of $B$ whenever that set contains a generically regular point. On the other hand, integral points on a toroidal curve lie in finitely many orbits of a torus, whence  the number of integral points of bounded height $B$ can only grow like a power of $\log B$. 
		It would be interesting to find out how `likely' it is that such a sparse set of points determines the approximation constant, and whether this has any implications for the geometry of the pair $(X,D)$.
        \end{enumerate}
    \end{remark}

   Finally, we also ask the following question, inspired by the notion of the \emph{essential} approximation constant, introduced and studied in depth by the first author~\cite{HuangIJNT,HuangAA,HuangBSMF,HuangANT}.
   \begin{question}\label{q:essential}
       In the setting of Conjecture \ref{conj:integralmckinnon}, does there exist a Zariski open subset $V$ and a (deformation) family $\{C_t\}_{t\in B}$ parametrized by $B$ of log rational or toroidal curves  sweeping out a Zariski dense subset of $X$ such that
	 	\begin{equation*}
			\alpha_v(x, \mU\cap C_t;L)=\alpha_v(x, \mU\cap V;L)
		\end{equation*}  
	   	for all $t\in B$?
   \end{question}

	\section{Example: a del Pezzo surface of degree 6}\label{se:ex}	
    As a first example, we shall investigate a del Pezzo surface over $k=\QQ$.    
	Let $X$ be the surface obtained by blowing up $\PP^2$ in three general points. This example is studied in detail in \cite{HuangIJNT}. It can equivalently be viewed as blowing up in $\PP^1\times\PP^1$ in two torus-invariant points with a specific configuration, say
	\begin{equation*}
		z_1=([0:1],\,[1:0])\quad\text{and}\quad z_2=([1:0],\,[0:1]).
	\end{equation*}
	We shall stick with this model and let $\mX$ be the toric $\ZZ$-scheme defined by the same fan as $X$. To fix further notation, Let $E_1$ and $E_2$ be the exceptional divisors above $z_1$ and $z_2$ in $X$; write $(s_0,s_1)$ and $(t_0,t_1)$ for the coordinate functions on the first and second factor of $\PP^1\times \PP^1$, respectively.
	
    Concretely, we shall investigate integral points on the complement of the boundary divisor $D$ that is the strict transform of $(s_0t_0=s_1t_1)\subset\PP^1\times\PP^1$, passing through both $z_1$ and $z_2$. As usual, write $\mD = \overline{D}$ for the Zariski closure of $D$ in $\mX$. As the class of the anticanonical line bundle is $-K_{X}=\cO(2,2)-E_1-E_2$ and the class of $D$ is $\cO(1,1)-E_1-E_2$, the log anticanonical class is $-K_X(D) = \cO(1,1)$, which can be checked to be big and nef, but not ample. Nevertheless, this surface proves to be a rich source of examples for the constructions in the previous sections. 
    
    Given a point $y=([s_0:s_1],\,[t_0:t_1])$ with coprime pairs of coordinates away from the center of the blow up, it is integral on $\mU$ if and only if 
	\begin{equation}\label{eq:intX3}
		s_0t_0-s_1t_1 \in \{\pm \gcd(s_0,t_1)\gcd(s_1,t_0)\}.
	\end{equation}
	Indeed, we may study the Cox ring
	\begin{equation*}
		R = \QQ[s_0',s_1',t_0',t_1',e_1,e_2]
	\end{equation*}
	of the toric variety $X$; the first four coordinates correspond to the strict transforms of the torus invariant lines under the orbit-cone correspondence, while the last two coordinates correspond to the two exceptional curves $E_1$ and $E_2$. Starting with $y$ as above, these new coordinates are obtained as
	\begin{equation*}
		e_1 = \gcd(s_0,t_1),\quad \text{and}\quad e_2 = \gcd(s_1,t_0),
	\end{equation*}
	making the strict transforms
	\begin{equation*}
		s_0' = e_1^{-1}s_0,\ s_1' = e_2^{-1}s_1,\ t_0' = e_2^{-1}t_0,\ \text{and}\ t_1' = e_1^{-1}t_1.
	\end{equation*}
	The integrality condition \begin{equation}\label{eq:intX3prime}
	    s_0't_0' - s_1't_1'\in \ZZ^\times, 
	\end{equation}
	when expressed in the Cox ring of $X$, translates to~\eqref{eq:intX3}.
\begin{lemma}
    The scheme $\mU$ satisfies the integral Hilbert property.
\end{lemma}
\begin{proof}
Let us consider the smaller open subscheme 
	\begin{equation*}
		\mV = \mU\setminus \overline{E_1 \cup E_2} = (\PP^1_\ZZ \times \PP^1_\ZZ) \setminus V(s_0t_0-s_1t_1).
	\end{equation*}
	It is enough to prove that $\mV(\ZZ)$ is dense in $\prod_{p<\infty}\mV(\ZZ_p)$, which implies that $\mV(\ZZ)$, and hence $\mU(\ZZ)$, is not thin. This follows from the effective Hilbert irreducibility theorem, as explained by Serre in \cite[\S3.6~Proof of Prop. 3.5.2]{SerreGalois}.
    
	Note that integral points on $\mV$ are those
	$([s_0:s_1],\ [t_0:t_1])$ 
	that satisfy
	\begin{equation}\label{eq:variant-integrality-cond}
		s_0t_0-s_1t_1 \in \ZZ^\times.
	\end{equation}
    Let $q > 1$ and fix a quadruple congruence residue $(a_0,a_1,b_0,b_1)\bmod q$ 
	satisfying \begin{equation*}
    a_0b_0-a_1b_1\equiv \pm 1\bmod q.
\end{equation*} Since such congruence residues form a topological basis of $\prod_{p<\infty}\mV(\ZZ_p)$, it is enough to show that $(a_0,a_1,b_0,b_1)$ can be lifted to an integral point of $\mV$. We can choose a coprime pair $(s_0',s_1')\in\ZZ^2$ such that $s_0'\equiv a_0\bmod q,s_1'\equiv a_1\bmod q$. Given such a pair $(s_0',s_1')$, we can select a pair $(r_0,r_1)\in\ZZ^2$ such that $$s_0'r_0-s_1'r_1=\pm 1.$$
	Now let $k\in\ZZ$ be such that
	\begin{equation*}
		k\equiv \pm(b_0r_1-b_1r_0)\bmod q;
	\end{equation*} 
    Note that the $\pm$ signs above are all the same. 
	Let the pair $(t_0',t_1')$ be defined by $t_0'=r_0-ks_1',t_1'=r_1-ks_0'$. One easily checks that the quadruple $(s_0',s_1',t_0',t_1')\in\ZZ^4$ satisfies the congruence condition 
	\begin{equation*}
		(s_0',s_1',t_0',t_1')\equiv (a_0,a_1,b_0,b_1)\bmod q
	\end{equation*}
	and the integrality condition \eqref{eq:variant-integrality-cond}. Hence $([s_0':s_1'],[t_0':t_1'])\in\mV(\ZZ)$. This finishes the proof.
\end{proof}

We now turn to studying integral approximation constants of the point
	\begin{equation*}
		x=([1:1],\,[1:1])\in D(\QQ).
	\end{equation*}
	We begin by describing the relevant height and distance functions explicitly. Any point $y$ on $X$ sufficiently close to $x$ cannot lie on the exceptional divisors, hence can be described by homogeneous coordinates
	\begin{equation*}
		y = ([s_0:s_1],\,[t_0:t_1])
	\end{equation*}
	on $\PP^1\times \PP^1$, both coordinates being represented by coprime integers. If such a point is sufficiently close to $x$, its distance can be computed through an affine chart:
	\begin{equation*}
		d(x,y) = d_\infty(x,y)\asymp \max\left(\left|\frac{s_1}{s_0}-1\right|,\left|\frac{t_1}{t_0}-1\right|\right).
	\end{equation*}
    Note that $-K_X(D)$ is not ample as the restrictions of $-K_X(D)$ to the two exceptional divisors are trivial. But defining a log anticanonical height function 
	\begin{equation*}
		H(y)=H_{-K_X(D)}(y)=\max\{|s_0t_0|,|s_0t_1|,|s_1t_0|,|s_1t_1|\}
	\end{equation*}
	through the standard choice of generators of $\cO(1,1)$
	(note that the $p$-adic variant of this maximum is always $1$ by the coprimality of the coordinates) and taking a point $x$ away from them, the Northcott property holds ``near $x$'', so that the log-anticanonical height is a good measure  of complexity in the definition of the approximation constant.

	In order to compute the approximation constant $\alpha_\infty(x,W;-K_X(D))$ for an appropriate $W$, we shall prove Liouville-type inequalities which allow to bound it from below and construct rational curves that provide matching bounds from above.
\begin{lemma}\label{lem:liouville-example}
	For all $y = ([s_0:s_1],\,[t_0:t_1])$ close to $x$ with $y\ne x$, we have 
	\begin{equation*}
		d(x,y)H(y)\gg 1.
	\end{equation*}
	Assume moreover that $s_0\neq s_1$ and $t_0\neq t_1$. Then 
	\begin{equation*}
		d(x,y)^2H(y)\gg 1.
	\end{equation*}
\end{lemma}
\begin{proof}
	For the first assertion, note that 
	\begin{equation*}
		d(x,y)H(y) \gg \left|\frac{t_1-t_0}{t_0}\right| |t_0|\max\{|s_0|,|s_1|\} \ge 1
	\end{equation*}
	if $t_1\ne t_0$, and the bound can be obtained symmetrically if $t_1=t_0$ (hence $s_1\ne s_0$).
	For the second bound, observe that
	\begin{equation*}
		d(x,y)^2H(y)\gg \left|\frac{(s_1-s_0)(t_1-t_0)}{s_0t_0}\right| |s_0t_0| \ge 1
	\end{equation*}
	under the stated condition.
\end{proof}

\begin{figure}
    \begin{center}
        \includegraphics[width=.75\linewidth]{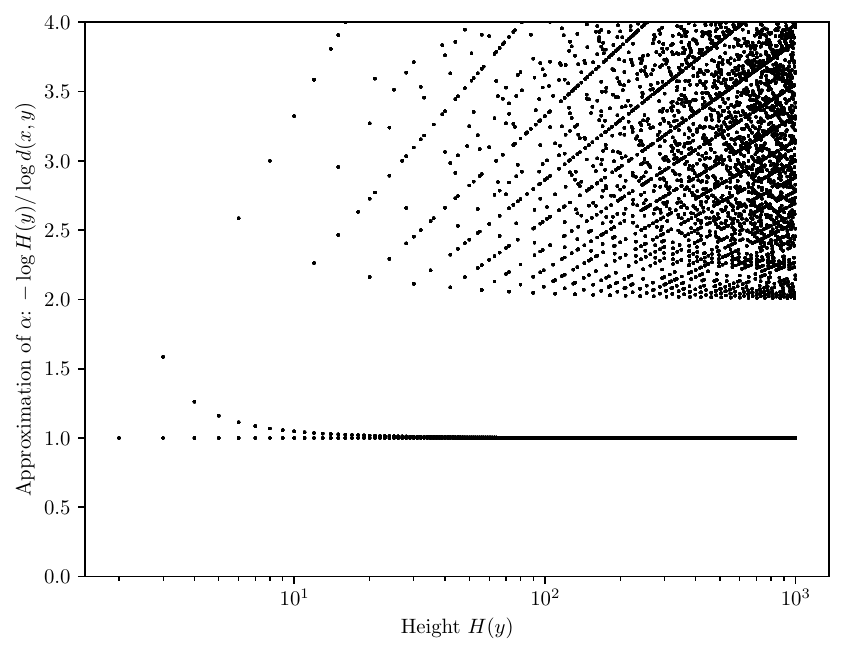}
        \caption{The heights $H(y)$ and ratios $-\log H(y) / \log d(x,y)$ of points of height at most $1000$. For points on a curve $\mC$ meeting $x$, the latter ratio converges to $\alpha(x,\mC_0;-K_X(D))$.}\label{fig:height-ratios}
    \end{center}
\end{figure}
We now determine all the possible integral curves that achieve the best approximation to $P_0$. Let $L_1$ and $L_2$ be the strict transforms of the lines $s_0=s_1$ and $t_0=t_1$ in $\PP^1\times\PP^1$, the classes of which belong to $\cO(1,0)$ and $\cO(0,1)$, respectively. They all have log anticanonical degree $1$, and intersection numbers with $D$ are also $1$ (at $x$). 
\begin{theorem}
    We have \begin{equation*}
	\alpha(x,\mU;-K_X(D))=\alpha(x,L_i\cap \mU;-K_X(D))=1
\end{equation*}
for $i\in\{1,2\}$. In particular, Conjecture~\ref{conj:integralmckinnon} holds for $X$.
\end{theorem}
\begin{proof}
    As $L_1$ and $L_2$ are smooth and contain the integral points $([1:1],\,[1:0])$ and $([1:0],\,[1:1])$, respectively, it follows from Proposition~\ref{prop:approximation-log-rational} that their integral approximation constants both equal one, and match the lower bound provided by the first part of Lemma~\ref{lem:liouville-example}.
\end{proof}

In order to tackle Question~\ref{q:essential}, we shall continue to construct rational curves achieving the `next best' approximation constants. Comparing heights and distances of points of height at most 1000 suggests that the resulting essential approximation constant is indeed~$2$ (Figure~\ref{fig:height-ratios}).
It turns out that three different types of curve achieve this: smooth log rational curves of degree two, nodal toroidal curves of degree four, and cuspidal log rational curves of the same degree.
We start with the first of these: smooth curves of log anticanonical degree two passing through $x$ such that they intersect $D$ tangently at $x$ --- this latter condition makes them log rational. The set of log rational curves of class $\CO(1,1)$ passing through $x$ tangently to $D$ is the pencil of curves of the form 
\[
	C_{a,b}\colon a (s_0t_0 -s_1t_1) +b(s_0-s_1)(t_0-t_1) = 0
\]
for coprime $(a,b) \in \ZZ^2$. Indeed, $s_0(t_0-t_1)$, $s_1(t_0-t_1)$, and $t_0(s_0-s_1)$ span the two-dimensional set of all divisors containing $x$; being tangent to $D$ in $x$ (i.e., being log rational) can be checked by computing partial derivatives and reduces the dimension by $1$; finally, the case $a=0$ has to be excluded as it corresponds to a reducible divisor (the union of $L_1$ and $L_2$), while with $b=0$, the curve coincides with the union of $D$ and the exceptional curves.
\begin{proposition}\label{prop:cab}
    If $a\ne 0$ and $b\ne 0$, then
	\begin{equation*}
		\alpha(x,C_{a,b}\cap \mU;-K_X(D))=2.
	\end{equation*}
\end{proposition}
\begin{proof} 
All of such curves $C_{a,b}$ have infinitely many integral points as they are everywhere locally soluble: for primes $p$ dividing $b$, the equation simplifies to $e_1e_2(s_0't_0'-s_1't_1')\equiv 0$ modulo $p$ in the Cox ring, having nonsingular solutions with $e_1=0$, for instance. For primes $p$ not dividing $b$, let $s_0=0$ and $s_1=t_1=1$ (such a point is automatically integral as $s_0t_0-s_1t_1\in \FF_p^\times$); the equation becomes $-a-b(t_0-1)=0$ and has a solution modulo $p$ that can be lifted by Hensel's Lemma.
Using Proposition~\ref{prop:approximation-log-rational}, we deduce that $\alpha(x,C_{a,b}\cap \mU;-K_X(D))=2$.
\end{proof}

In \cite[\S5.2.2]{HuangAA}, a family of nodal curves at $x$ is constructed:
\begin{equation}\label{eq:nodal}
   S_{a,b}\colon as_0s_1(t_0-t_1)^2=b t_0t_1(s_0-s_1)^2,
\end{equation}
where the parameters $(a,b)$ are coprime non-zero integers.
In addition to $x$, these curves meet the curve $s_0t_0=s_1t_1$ in $\PP^1\times \PP^1$ only in $z_1$ and $z_2$; as they are not tangent to it in these points, they only meet $D$ in $x$. Thus, the curves $S_{a,b}$ are nodal toroidal. 
\begin{proposition}\label{prop:sab}
    If $ab>0$ and $ab$ is not a square, then
	\begin{equation*}
	\alpha(x,S_{a,b}\cap \mU;-K_X(D))=2.
	\end{equation*}
\end{proposition} 
\begin{proof}

    Each of these curves $S_{a,b}$ has log anticanonical degree $4$, and admits the trivial integral point $([0:1],\,[0:1])$, so we are left to study their splitting fields. The node at $x$ has slope $\pm\sqrt{b/a}$, which is real quadratic if and only if $ab$ is positive and not a square.  In precisely those cases, the splitting field $\QQ(\sqrt{b/a})$ of the toroidal curve has an infinite group of units, hence the curve contains infinitely many integral points by Corollary~\ref{cor:torus-infinite-numberfield}. The deduction of the integral approximation constant follows from Proposition~\ref{prop:approximation-toroidal}.
\end{proof}
We remark that $\alpha(x,S_{a,b}\cap \mU;-K_X(D))=4$ if $ab>0$ and $ab$ is a square, and is infinite in all other cases, as the remaining curves only contain finitely many points each.

In \cite[\S3.4.1]{HuangBSMF}, a family of cuspidal rational curves, all with a cuspidal singularity at $x$, is constructed:
\begin{equation}\label{eq:cusp}
	R_{a,b}\colon a^2(t_1-t_0)^2s_0s_1+b^2(s_1-s_0)^2t_0t_1-2abs_1t_1(s_1-s_0)(t_1-t_0)=0,
\end{equation}
where the parameters $(a,b)$ are coprime non-zero integers, distinct from $(1,-1)$ and $(-1,1)$. Observe that the cases $(a,b)=(1,-1)$ or $(-1,1)$ correspond to reducible curves.
\begin{proposition}\label{prop:rab}
    If $(a,b)\neq(1,-1)$ or $(-1,1)$, then
    \begin{equation*}
        \alpha(x,R_{a,b}\cap \mU;-K_X(D))=2.
    \end{equation*}
\end{proposition} 
\begin{proof}
The parameters $(a,b)$ for this family represent the slope $a/b$ of $R_{a,b}$ at $x$. 
Its intersection multiplicity with the boundary can be computed to be  $\langle R_{a,b},D\rangle=2$ by noting that it meets the image of the boundary in $\PP^1\times \PP^1$ in $x$, $z_1$ and $z_2$, which means that $R_{a,b}$ intersects $D$ precisely at one point $x$, with multiplicity $2$, due to the cusp. In particular, these curves are log rational of degree four, making their integral approximation constants $2$ by Proposition~\ref{prop:approximation-log-rational}, as each of them contains the trivial integral point $([1:0],[1:0])$. 
\end{proof}  

\begin{figure}
\begin{center}
    \begin{minipage}{.5\textwidth}
    \begin{center}
        \includegraphics[width=.9\linewidth]{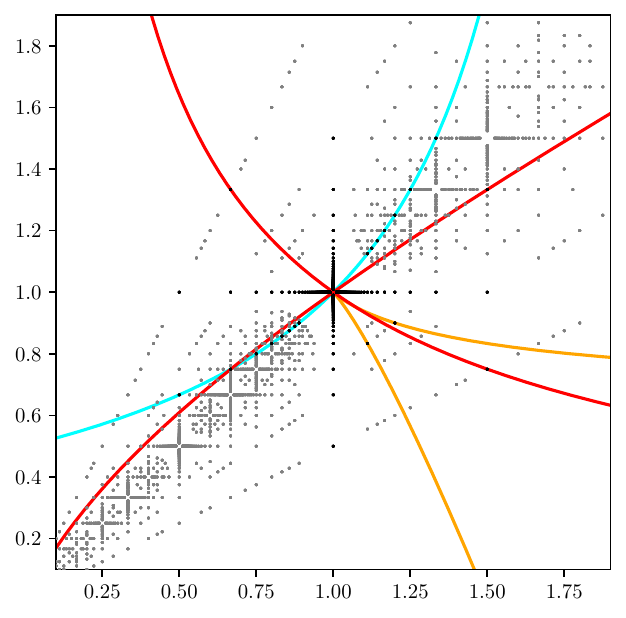} 
    \end{center}
    \end{minipage}%
    \begin{minipage}{.5\textwidth}
    \begin{center}
        \includegraphics[width=.9\linewidth]{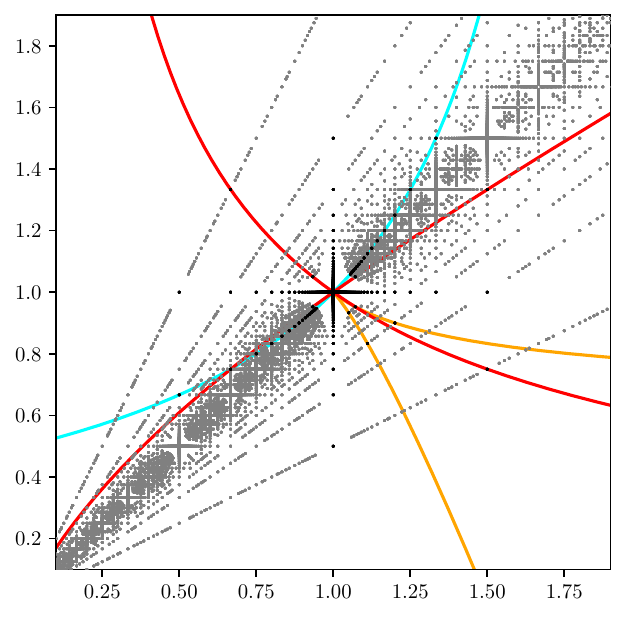}
    \end{center}
    \end{minipage}

    \caption{Points whose height is bounded by $100$ (left) and $350$ (right), respectively, shown along a chart of $
    \PP^1\times \PP^1$ around $x$. The best approximants lie on the two `axes', achieving the approximation constant $1$. A smooth and a singular log rational curve, as well as a nodal toroidal curve achieving the essential approximation constant $2$ are shown in cyan, orange, and red, respectively.}\label{fig:curves}
    \end{center}
\end{figure}

Togogether with Lemma~\ref{lem:liouville-example}, this shows that the essential approximation constant is attained on rational curves --- simultaneously on smooth log rational, cuspidal log rational, and nodal toroidal curves (Figure~\ref{fig:curves}).

\begin{theorem}
We have $$\alpha(x;\mU \setminus (L_1\cap L_2);-K_X(D))=2.$$
    Moreover, Question \ref{q:essential} has an affirmative answer for $X$.
\end{theorem}
\begin{proof}
 We use Propositions~\ref{prop:cab}, \ref{prop:sab}, and \ref{prop:rab} as upper bounds and the second half of Lemma~\ref{lem:liouville-example} for the matching lower bound.
    Finally, to answer Question \ref{q:essential}, we note that the union of any one of the three families $(C_{a,b})$, $(S_{a,b})$, and $(R_{a,b})$ forms a Zariski dense subset (though the union of rational points on the cuspidal curves $R_{a,b}$ forms a thin set).
\end{proof}

\begin{remark}
For any $W\subset X$, and $B\gg 1$, let us define 
$$N(W;B)\coloneq \#\{P\in\mU(\ZZ)\cap W:H(P)\leq B\}.$$
Then, using the parametrizations obtained in Section~\ref{se:ratcurve}, it is straightforward to determine
\begin{align*}
N(L_i;B)    &\asymp B && \text{for $i\in \{1,2\}$,}\\
N(C_{a,b};B)&\asymp_{a,b} B^\frac{1}{2} && \text{for coprime $(a,b)\in \ZZ^2$ with $ab\ne 0$,}\\
N(S_{a,b};B)&\asymp_{a,b} \log B&&\text{for coprime $(a,b)\in \ZZ^2$ with $ab>0$ not a square,}\\
 N(R_{a,b};B)&\asymp_{a,b} B^{\frac{1}{4}} && \text{for coprime $(a,b)\in \ZZ^2$ with $ab\neq0,a+b\ne 0$.}
\end{align*}

We thus observe that, although the last three families all achieve the same approximation constant, the number of points traced on them exhibits quite different orders of magnitude.
\end{remark}

\section{Toric varieties}\label{se:toric}

	Let $X$ be the smooth projective split toric variety over $k$ associated with a regular complete fan $\Sigma$ in $\RR^n$ admitting at least one strictly convex piecewise linear function,
    and let $\mX$ the corresponding toric $\fo$-scheme. 
    We write $\Sigma^{\max}$ for the set of all $n$-dimensional cones (called \emph{maximal cones}). A one-dimensional cone $\rho$ is called a \emph{ray}, with a primitive integral generator $n_\rho$ in $\RR^n$ and corresponding torus-invariant divisor $D_\rho \subseteq X$;  we write $r=\operatorname{rank}\operatorname{Pic}(X)$. The set of all rays is denoted by $\Sigma(1)$, and $\#\Sigma(1)=n+r$. 
	The morphism
	\begin{equation*}
		\phi\colon \RR^{n+r} = \RR^{\Sigma(1)}\to \RR^n,\quad \rho\mapsto n_\rho,
	\end{equation*}
	and its dual induce short exact sequences
	\begin{equation}\label{eq:lattice-sequences}
		\begin{tikzcd}[row sep = tiny]
			0 \ar[r] & N_1(X)_\RR \ar[r] & \RR^{\Sigma(1)} \ar[r, "\phi"] & \RR^n \ar[r]& 0 & \text{and} \\
			0 \ar[r] & \RR^n\ar[r,"\phi^*"] & \RR^{\Sigma(1)} \ar[r] & \Pic(X)_\RR \ar[r] & 0;
		\end{tikzcd}
	\end{equation}
	writing $T_\NS = \widehat{\Pic(X)}\cong \Gm^r$ for the Néron--Severi $k$-torus and $T = \Gm^n$ for the open orbit of $X$, restricting the latter sequence to the character lattices induces one of split tori:
	\begin{equation}\label{eq:short-ex-tori}
		\begin{tikzcd}
			1\ar[r] & T_\NS \ar[r]& \Gm^{\Sigma(1)} \ar[r] & T \ar[r]& 1.
		\end{tikzcd}
	\end{equation}

	Let $Y$ be the quasi-affine toric variety constructed by pulling back the fan $\Sigma$ along $\phi$, called the \emph{principal universal torsor} after Salberger~\cite{Salberger}: it is a $T_\NS$-torsor via the restriction of the torus action along~\eqref{eq:short-ex-tori} and the morphism $\pi\colon Y\to X$ induced by $\phi$. Similarly, the $\fo$-scheme $\mY$ associated with the same fan is a $\mT_\NS$-torsor over $\mX$, where $\mT_\NS\cong \Gmok^r$ is the $\fo$-torus dual to $\Pic(X)$.

	We recall the following notion that plays a key role in Batyrev's classification of higher-dimensional toric varieties \cite{Batyrev}.
	A collection of rays $\cP=\{\rho_1,\dots,\rho_s\}$ is called a \emph{central primitive collection} if
	\begin{equation*}
		\sum_{i=1}^{s} n_{\rho_i}=0,
	\end{equation*}
	the rays $\rho_1,\dots,\rho_s$ do not generate a cone of $\Sigma$, and all proper subsets of them do generate one. They give rise to \emph{minimal rational curves} (in the sense of deformation theory): as elements of the kernel of $\phi$, they are elements of the dual lattice to $\Pic(X)$, and given a torus-invariant divisor $L = \sum_{\rho\in \Sigma(1)} a_\rho D_\rho$, using~\eqref{eq:lattice-sequences}, the $L$-degree of any associated curve is simply computed to be
	\begin{equation*}
		\delta_{\cP} = \sum_{\rho\in \cP} a_\rho,
	\end{equation*}
	which is an integer also depending on $L$.
	See \cite[\S\,2B]{HuangANT} for more details, for instance. 
	
	In order to study integral points on a log pair, we need a torus-invariant divisor. For any nonempty subset of rays $\cA\subseteq \Sigma(1)$, we denote by  $D_{\cA}=\sum_{\rho\in\cA } D_\rho$ its corresponding divisor in $X$, and by $\mD_{\cA}=\sum_{\rho\in\cA } \mD_\rho$ the corresponding divisor in $\mX$. Moreover, denote by $U_{\cA}=X\setminus D_\cA$ the corresponding open toric subvariety and $\mU_{\cA}=\mX \setminus \mD_{\cA}$ the corresponding open toric subscheme. 
    
In previous work~\cite[Thm.~1.2]{HuangANT}, the first author proved that the rational approximation constant on the open orbit equals the minimal degree among all minimal rational curves. We shall compute the integral approximation constant under the following hypotheses for a family of log pairs $(X,D_\cA)$ of a special shape: 

\begin{theorem}\label{thm:toricgeneralrestrictive}
   	Let $X$ be a smooth, projective toric variety whose pseudoeffective cone is simplicial with corresponding toric $\fo$-scheme $\mX$, $L$ be an ample line bundle, and $v$ be an archimedean place.
	Moreover, let $\cA$ consist of a single element of a central primitive collection $\cP$, and  
    let $\delta_\cP$ be the $L$-degree of any minimal rational curve corresponding to this collection.
	
	If $x\in D_{\cA}^\circ(k)=D_{\cA}(k)\setminus \cup_{\rho\notin\cA}D_{\rho}(k)$, then
		\begin{equation*}
			\alpha_v(x,\mU_{\cA};L)=\delta_\cP.
		\end{equation*}
\end{theorem}
In particular, this confirms Conjecture~\ref{conj:integralmckinnon} for the log pair $(X,D_\cA)$ and the choice of $x\in D_{\cA}^\circ(k)$.

\begin{remark}\label{rmk:toric}
\begin{enumerate}
    \item Note that, as every ray can be contained in at most one central primitive collection~\cite[Lem.~6.4]{HuangANT}, the degree $\delta_\cP$ is unique.   
        \item Note that removing $\cA$ from $\Sigma$ preserves at least one (smooth) maximal cone, corresponding to an open toric subscheme isomorphic to $\dA_\fo^n$ of $\mU_\cA$. In particular, $\mU_{\cA}(\fo)$ is not thin.
        Moreover, as $-K_X(D_{\CA})$ remains big, it follows from work of Santens~\cite[Thm.~1.1]{Santens} that $\mU_{\cA}(\fo)$ is equidistributed.
        \item\label{enum:simplicial-eff} That the pseudoeffective cone is simplicial has a simple combinatorial description: it holds if and only if there exists a maximal cone $\sigma_0\in\Sigma^{\max}$ such that all the generators in $\Sigma(1)\setminus\sigma_0(1)$ are linear combinations of those in $\sigma_0(1)$ with negative integer coefficients~\cite[Lem.~6.3]{HuangANT}.
    \item More generally, for every central primitive collection $\cP$ we can fix a ray $\rho_{\cP}\in\cP$, and let $$\cA=\bigcup_{\cP\text{ central primitive}}\{\rho_{\cP}\}\subseteq \Sigma^{(1)}.$$ Let $A\subseteq \cA$ be a maximal face of the Clemens complex (that is, a maximal subset of $\cA$ that generates a cone in $\Sigma$),  
         let $Z_{A}=\bigcap_{\rho\in A}D_{\rho}$, and let
		\begin{equation*}
			x\in Z_A^\circ(k)=(Z_A\setminus\bigcup_{\rho\notin A}D_{\rho})(k).
		\end{equation*}
		Then a similar argument to the one that is to follow shows that
		\begin{equation*}
			\alpha_v(x,\mU_{\cA};L)=\min\left\{\deg_L(C)\ \middle|\ 
			\substack{C\text{ is minimal and has class }\cP, \\ \cP\cap A\neq\emptyset} 
			\right\}.
		\end{equation*}
		 For simplicity of exposition, we content ourselves with the more restrictive version of Theorem~\ref{thm:toricgeneralrestrictive} above.
\end{enumerate}
    
\end{remark}

\begin{proof}[Proof of Theorem~\ref{thm:toricgeneralrestrictive}]
	For the rest of this section, the data mentioned in the theorem (in particular $\mX$, $x$, and $L$) are treated as fixed and all implied constants may depend on them.
	We begin by noting that only integral points on the open orbit $T = X\setminus \bigcup_{\rho\in \Sigma(1)}D_\rho$ can contribute to the approximation constant, that is,
	\begin{equation*}
        \alpha_v(x,\mU_{\cA};L) = \alpha_v(x,\mU_{\cA}\cap T;L).
    \end{equation*}
	Indeed, the function
	\begin{equation*}
		X(k)\setminus T \to \RR_{\geq 0},\quad y\mapsto d(x,y)
	\end{equation*}
	is continuous and thus attains its minimum on the compact set $\bigcup_{\rho \not \in \cA} D_\rho$. This minimum is nonzero as the set does not contain $x$ by assumption.
	Moreover, for $\rho\in \cA$, the divisor $D_\rho$ does not contain any integral points on $\mU_\cA = \mX \setminus \mD_\rho$; hence, integral points outside $T$  cannot contribute to the approximation constant.

	For the rest of the argument, we basically follow the blueprint of \cite[\S\,6]{HuangANT}, but need to take into account the integrality condition imposed by $\cA$.
	Let $\sigma_0\in\Sigma^{\max}$ be a maximal cone as in Remark~\ref{rmk:toric}~(\ref{enum:simplicial-eff}). We may label the rays $\Sigma(1) = \{\rho_1,\dots,\rho_{n+r}\}$ so that $\sigma_0(1)=\{\rho_1,\dots,\rho_n\}$; for the remaining rays, we obtain a relation
	\begin{equation}\label{eq:relationR}
	\CR_{j}\colon n_{\rho_{n+j}}+\sum_{i=1}^{n} b_{i,j}n_{\rho_i}=0,\quad 1\leq j\leq r
	\end{equation}
	between rays with positive coefficients $b_{i,j}\in\ZZ_{\geq 0}$ for $1\leq i\leq n$ by the same remark.

	Appealing to Salberger's seminal work \cite{Salberger}, rational points on split toric varieties can be parametrized in terms of Cox coordinates of universal torsors. We shall begin by sketching some properties of this parametrization and refer the reader to \cite[\S\,3, \S\,6B]{HuangANT} for more details. Let $\cC$ be a set of representatives of the ideal class group of $k$, each of them an ideal of $\fo$.
	Then there exists a family $(\pi_{\bc}:\mY_{\bc}\to \mX)_{\bc\in \cC^r}$ of \emph{twisted universal torsors} such that
	\begin{equation*}
		X(k)=\mX(\fo)=\bigcup_{\bc\in\cC^r}\pi_{\bc}(\mY_{\bc}(\fo)).
	\end{equation*}

	Each element of this family is a $\mT_\NS$-torsor over $\mX$, and its generic fiber is isomorphic to the quasi-affine principal universal torsor $Y$.
	The relations \eqref{eq:relationR} being equivalent to 
    \begin{equation}\label{eq:picbasis}
        [D_{\rho_i}]=\sum_{j=1}^r b_{i,j} [D_{\rho_{n+j}}] \quad \text{in }\operatorname{Pic}(X),\ 1\leq i\leq n,
    \end{equation}
    the classes of the divisors $D_{\rho_{n+1}},\cdots,D_{\rho_{n+r}}$ form a basis of $\Pic(X)$, which we shall work in.
	For a divisor $D$ of class $[D]=\sum_{j=1}^{r}d_j[D_{\rho_{n+j}}]\in\Pic(X)$ and a tuple $\bc=(c_1,\cdots,c_r)\in \cC^r$, we write $\bc^D$ for the ideal $\prod_{j=1}^{r}c_j^{d_j}$.
	Each $\mY_{\bc}(\fo)$ can be viewed as the set
	\begin{equation*}
		\bigoplus_{\rho\in\Sigma(1)} \bc^{D_{\rho}}\subset\fo^{\Sigma(1)}
	\end{equation*}
	under the embedding
	\begin{equation*}
		 \mY_{\bc}(\fo) \subset Y(k)\subset k^{\Sigma(1)}.
	\end{equation*}
	Every $P\in X(k)$ can be lifted to an integral point
	\begin{equation*}
		\bX(P) = (X_1(P),\dots, X_{n+r}(P))\in \mY_{\bc}(\fo)
	\end{equation*}
	that is unique modulo the action of the Néron--Severi torus $\mT_{\operatorname{NS}}$,
	for a $\bc\in\cC^r$ that is uniquely determined by $P$.

	For any $\bx\in k^{\Sigma(1)}$ and torus invariant divisor $D=\sum_{\rho\in\Sigma(1)}a_{\rho} D_{\rho}$, we write $\bx^{D}$ for the product $\prod_{\rho\in\Sigma(1)}x_{\rho}^{a_{\rho}}$ whenever it is well-defined. 
	The height of a rational point can be described explicitly by means of its lift as above.
	The line bundle $L$ is globally generated since it is assumed to be ample. Write $L=\CO_X(D)$ for a torus-invariant divisor $D$. Then by \cite[Prop.~3.4]{HuangANT},
    \begin{equation}\label{eq:height}
        H_L(P)\asymp 
        \prod_{v\in\infty_k} \sup_{\sigma\in\Sigma^{\max}}\left|\bX(P)^{D(\sigma)}\right|_{v},
    \end{equation}  where for any $\sigma\in\Sigma^{\max}$, let $D(\sigma)$ be the divisor that is linearly equivalent to $D$ and whose support lies within $\bigcup_{\rho\in \Sigma(1)\setminus\sigma(1)} D_{\rho}$ (cf. \cite[Eq.~(28)]{HuangANT}).
    
        By assumption, there is a central primitive collection $\cP$ such that $\cA\subset\cP$, and we fix one now. Then by \cite[Lem.~6.6]{HuangANT}, $\#\cP\setminus\sigma_0(1)=1$. So we may relabel the rays so that $\cP=\{\rho_1,\cdots,\rho_{t},\rho_{n+1}\}$ for a certain $1\leq t\leq n$. Let  $\sigma_1,\dots,\sigma_t\in\Sigma^{\max}$, labeled such that
	\begin{equation*}
		\sigma_{i}\cap \sigma_0
		= \RR_{\geq 0}\rho_1+\cdots +
			\widehat{\RR_{\geq 0}\rho_{i}}+\cdots+\RR_{\geq 0} \rho_n,\quad 1\leq i\leq t.
	\end{equation*} Here the notation $\widehat{}$ means that the term is removed from the sum.
By \cite[Lem.~6.6]{HuangANT}, we have
\begin{equation}\label{eq:sigmai}
    \sigma_{i}=\RR_{\geq 0}\rho_1+\cdots +
\widehat{\RR_{\geq 0}\rho_{i}}+\cdots+\RR_{\geq 0} \rho_n+\RR_{\geq 0} \rho_{n+1},\quad 1\leq i\leq t.
\end{equation}
        
We begin with the case $\cA\subset \sigma_0(1)$ and may assume $\cA=\{\rho_1\}$. 
Let $U_{\sigma_0}\subset X$ be the open affine toric subvariety induced by the maximal cone $\sigma_0$; we shall use its parametrization
	\begin{equation}\label{eq:parahyp}
		\begin{split}
			\pi\colon \pi^{-1}U_{\sigma_0}&\longrightarrow U_{\sigma_0}\\
			(X_1,\cdots,X_{n+r})&\longmapsto(y_1,\dots,y_n)= \left(\frac{X_1}{\prod_{j=1}^r
				X_{n+j}^{b_{1,j}}},\dots,\frac{X_n}{\prod_{j=1}^r
				X_{n+j}^{b_{n,j}}}\right)
		\end{split}
	\end{equation}
(cf.~\cite[\S\,6B1]{HuangANT}).

 By assumption, $x\in D_{\cA}^{\circ}(k)$, so that
\begin{equation*}
	x=(0,y_2(x),\dots,y_n(x))\in U_{\sigma_0}
\end{equation*}
with $y_i(x)\neq 0$ for $2\leq i\leq n$. Moreover, 
for any point $P\in (U_{\cA}\cap U_{\sigma_0})(k)$, every lift $\bX(P)=(X_1(P),\cdots,X_{n+r}(P))\in \mY_{\bc}(\fo)\subset\fo^{n+r}$ for an appropriate $\bc\in\cC^r$ satisfies $X_1(P)\neq 0$ and $X_{n+j}(P)\neq 0$ for 1$\leq j\leq r$. Fixing a place $v\in\infty_k$,
by~\eqref{eq:distance_A1}, a lower bound for any fixed $v$-adic distance function is
	\begin{equation}\label{eq:dist}
		d_v(P,x)\gg \min\left(1,\left|\frac{X_1(P)}{\prod_{j=1}^{r}X_{n+j}(P)^{{b}_{1,j}}}\right|_v\right).
	\end{equation}

Recall from \eqref{eq:sigmai} that the maximal cone $\sigma_1$, without the ray $\rho_1$ and adjacent to $\sigma_0$, contains the ray $\rho_{n+1}$. It gives rise to the divisor $D(\sigma_1)$. Let $c_1$ be the multiplicity of $D_{\rho_1}$ in $D(\sigma_1)$. Then the computation in \cite[(54)]{HuangANT} shows that $$c_1=\delta_\cP.$$  According to \cite[Lem.~8.9, Rem.~11.23]{Salberger}, together with the divisor $D(\sigma_0)$, we have $$\frac{\bX(P)^{D(\sigma_1)}}{\bX(P)^{D(\sigma_0)}}=\left(\frac{X_1(P)}{\prod_{j=1}^{r}X_{n+j}(P)^{{b}_{1,j}}}\right)^{\delta_{\cP}}.$$ Invoking \eqref{eq:height}, we obtain the bound 
\begin{equation}\label{eq:case1}
    \begin{split}
        H_L(P)&\gg \prod_{v'\in\infty_k\setminus\{v\}}\left|\bX(P)^{D(\sigma_1)}\right|_{v'}\times \left|\bX(P)^{D(\sigma_0)}\right|_{v}\\ &\gg \left|\frac{\bX(P)^{D(\sigma_0)}}{\bX(P)^{D(\sigma_1)}}\right|_{v}=\left|\frac{X_1(P)}{\prod_{j=1}^{r}X_{n+j}(P)^{{b}_{1,j}}}\right|_{v}^{-\delta_{\cP}}
    \end{split} 
\end{equation}
for the height function, where for the second inequality we use the product formula, all the implied constants being uniform. From this and \eqref{eq:dist} we easily deduce that $$H_L(P)d_v(P,x)^{\delta_\cP}\gg 1.$$  This shows that $$\alpha_v(x,\mU_{\cA}\cap U_{\sigma_0};L)\geq \alpha_v(x,U_{\cA}\cap U_{\sigma_0};L)\geq \delta_\cP.$$

 To get an upper bound, we now lift $x\in D_{\cA}^{\circ}(k)$ to
 \begin{equation*}
	\bX(x)=(0,x_2,\cdots,x_{n+r})
    \in\mY_{\bc(x)}(\fo)\subset \fo^{n+r}
 \end{equation*}
 	for a fixed $\bc(x)\in \cC^r$. 
    Let $\varepsilon\in \fo^\times$ be fixed. We consider the rational curve $\phi\colon \PP^1\to C\subset X$ defined by 
    \begin{equation}\label{eq:primitive-curve-param}
        (0,0)\neq(u,v)\mapsto \pi(u\varepsilon,vx_2,\cdots,vx_t,x_{t+1},\cdots,x_n,vx_{n+1},x_{n+2},\cdots,x_{n+r}).
    \end{equation}
    Let $\mC$ be the Zariski closure of $C$ in $\mX$.
Note that $\phi(u,v)\in D_\cA$ precisely if the first entry on the right-hand side of~\eqref{eq:primitive-curve-param} vanishes, that is, if $u=0$, and this intersection point is $C\cap D_\cA = \{x\}$. In particular, $C\not\subset D_{\cA}$. Moreover, as $x$ does not lie in any of the remaining torus-invariant divisors, the curve $C$ cannot be contained in any of these divisors, so that it meets the open orbit.  By \cite[Thm.~2.2]{HuangANT}, $C$ is a (smooth) minimal rational curve
corresponding to $\cP$.
Let $C_0=C\cap U_{\cA}$ be the open curve and let $\mC_0=C\cap\mU_{\cA}$.
The point $$P_x=\pi(\varepsilon,0,\cdots,0,x_{t+1},\cdots,x_n,0,x_{n+2},\cdots,x_{n+r})\in C_0(k)$$ lies in  $\bigcap_{\rho\in\cP\setminus\{\rho_1\}}D_{\rho}$ and admits a lift $\bX(P_x)=(X_i(P_x))_{1\leq i\leq n+r}\in \mY_{\bc(x)'}(\fo)$ for a certain $\bc(x)'\in\cC^r$. 

We recall that the principal universal torsor $Y$ is the complement in $\dA^{\Sigma(1)}$ of union of all sets of the form
\begin{equation*}
	\bigcap_{\rho\in\cI}\{X_{\rho}=0\}
\end{equation*}
where $\cI\subset\Sigma(1)$ ranges over all sets of rays such that
$\bigcap_{\rho\in\cI} D_{\rho}=\emptyset$, that is, such that the rays do not generate a cone in $\Sigma$ (cf. \cite[Eq.~(22)]{HuangANT}). In particular, the primitive collection
$\cP = \{\rho_1,\cdots,\rho_t,\rho_{n+1}\}$ is such a set, so that the corresponding divisors do not meet; hence, the Cox coordinates $X_1,\dots, X_t, X_{n+1}$ are not allowed to vanish simultaneously for rational points in $Y(k)$ and must be coprime for any integral point on $\mY$, which in turn translates to the coprimality condition 
\begin{equation*}
	X_{n+1}(P_x)\bc(x)'^{-D_{\rho_{n+1}}}+\sum_{1\leq i\leq t} X_i(P_x)\bc(x)'^{-D_{\rho_i}}=\fo
\end{equation*}
on the twisted torsors $\mY_{\bc}$ as in \cite[Thm.~2.7 (iii)]{FP}: it means that the corresponding point avoids the vanishing locus of the \emph{twisted ideal}
\begin{equation*}
	(X_1,\dots,X_t,X_{n+r})_{\bc}
\end{equation*}
as defined by Frei and Pieropan~\cite[Def.~2.4]{FP}.
But $X_{n+1}(P_x)=X_2(P_x)=\cdots=X_{t}(P_x)=0$, so that
	\begin{equation*}
		X_1(P_x)\bc(x)'^{-D_{\rho_1}}=\fo.
	\end{equation*}
	This means that $\bX(P_x)$ in fact avoids the vanishing locus of the twisted ideal $(X_0)_{\bc(x)'}$, that is, the preimage of $\mD_{\rho_0}$ in $\mY_{\bc}$ (that this is in fact the preimage is analogous to the identity \cite[Eq.~(2.13)]{MR4822120}).
	From this, we may deduce that its image $P_x$ in fact lies in lies in the image $\mU_{\cA} = \mX\setminus \mD_{\rho_0}$. 
	We have thus obtained a regular integral point on $\mC_0$, and it follows from Proposition~\ref{prop:a1para} that $$\alpha_v(x,\mU_{\cA}\cap U_{\sigma_0};L)\leq\alpha_v(x,\mC_0\cap U_{\sigma_0};L)=\alpha_v(x,\mC_0;L)=\delta_\cP.$$

We next address the case $\cA=\{\rho_{n+1}\}\not\subset \sigma_0(1)$. 
We have seen that the maximal adjacent cone $\sigma_1$ satisfies $\rho_{n+1}\in\sigma_1(1)$. 
That $\cP$ is central primitive means that 
\begin{equation*}
	n_{\rho_1}=-\sum_{i=2}^{t}n_{\rho_i}-n_{\rho_{n+1}};
\end{equation*}
using this, we may rewrite the relations $\CR_j$ with $2\le j\le r$ as in~\eqref{eq:relationR} as
\begin{align*}
    n_{\rho_{n+j}}&=-\sum_{\substack{2\leq i\leq n}}b_{i,j}n_{\rho_i}+b_{1,j}\left(\sum_{\substack{2\leq i\leq t}}n_{\rho_i}+n_{\rho_{n+1}}\right)\\ &=-\sum_{\substack{2\leq i\leq t}}(b_{i,j}-b_{1,j})n_{\rho_i}-\sum_{t+1\leq i\leq n} b_{i,j}n_{\rho_i}+b_{1,j}n_{\rho_{n+1}}.
\end{align*}
 For any $P\in U_{\sigma_1}(k)$, let $$(y_2'(P),\cdots,y_{n}'(P),y_{n+1}'(P))$$ be the coordinates under the parametrization of $U_{\sigma_1}$. We then have
\begin{equation*}
	y_{n+1}'(P)=\frac{\prod_{1\leq j\leq r}X_{n+j}(P)^{b_{1,j}}}{X_1(P)}
\end{equation*}  
for  any lift $\bX(P)$.
Now $x\in U_{\sigma_1}(k)$ and $y_{n+1}'(x)=0$.
For any point $P$ in $(U_{\sigma_1}\cap U_{\cA})(k)$, we have $y_{n+1}'(P)\neq 0$. Hence, fixing a place $v\in\infty_k$, a lower bound of the relevant $v$-adic distance function is given by
\begin{equation}\label{eq:dist2}
	d_v(P,x)\gg \min\left(1,\left|\frac{\prod_{1\leq j\leq r}X_{n+j}(P)^{b_{1,j}}}{X_1(P)}\right|_v\right).
\end{equation} 
By \eqref{eq:dist2}, based on a variant of \eqref{eq:case1}, we deduce that \begin{align*}
H_L(P)&\gg \prod_{v'\in\infty_k\setminus\{v\}}\left|\bX(P)^{D(\sigma_0)}\right|_{v'}\times \left|\bX(P)^{D(\sigma_1)}\right|_{v}\\ &\gg \left|\frac{\bX(P)^{D(\sigma_1)}}{\bX(P)^{D(\sigma_0)}}\right|_{v}=\left|\frac{X_1(P)}{\prod_{j=1}^{r}X_{n+j}(P)^{{b}_{1,j}}}\right|_{v}^{\delta_{\cP}},\end{align*} from which we again obtain $$H_L(P)d_v(P,x)^{\delta_\cP}\gg 1.$$ 
These all together imply that $$\alpha_v(x,\mU_{\cA}\cap U_{\sigma_1};L)\geq \alpha_v(x,U_{\cA}\cap U_{\sigma_1};L)\geq\delta_\cP.$$

Finally, we lift $x$ to
\begin{equation*}
	\bX(x)=(x_1,\cdots,x_t,x_{t+1},\cdots,x_n,0,x_{n+2},\cdots,x_{n+r}),
\end{equation*}
and let $C$ be the rational curve defined by 
\begin{equation*}
	(0,0)\neq(u,v)\mapsto \pi(vx_1,\cdots,vx_t,x_{t+1},\cdots,x_n,u\varepsilon,x_{n+2},\cdots,x_{n+r}),
\end{equation*}
in $X$, where $\varepsilon\in\fo\setminus\{0\}$ is fixed; let $\mC=\overline{C}$ be the Zariski closure in $\mX$, and let $\mC_0=\mC\cap \mU_{\cA}$. Then, as before, we verify that $C$ is minimal and has class $\cP$, meets $D_\cA$ precisely in $x$, and the point
\begin{equation*}
	\pi(0,\cdots,0,x_{t+1},\cdots,x_n,\varepsilon,x_{n+2},\cdots,x_{n+r})
\end{equation*}
lifts to a regular point of $\mC_0(\fo)$, whence 
\begin{equation*}
	\alpha_v(x,\mU_{\cA}\cap U_{\sigma_t};L)\leq\alpha_v(x,\mC_0\cap U_{\sigma_t};L)=\alpha_v(x,\mC_0;L)=\delta_\cP.
\end{equation*}
The proof is thus completed.
\end{proof}

\bibliographystyle{plain}

\begin{thebibliography}{10}
\footnotesize{
\bibitem{Batyrev}
V.~V. Batyrev.
\newblock On the classification of smooth projective toric varieties.
\newblock {\em Tohoku Math. J. (2)}, 43(4):569--585, 1991.

\bibitem{MR1045822}
S.~Bosch, W.~L\"utkebohmert, and M.~Raynaud.
\newblock {\em N\'eron models}, volume~21 of {\em Ergebnisse der Mathematik und ihrer Grenzgebiete (3)}.
\newblock Springer-Verlag, Berlin, 1990.

\bibitem{MR1191735}
F.~Campana.
\newblock Connexit\'e{} rationnelle des vari\'et\'es de {F}ano.
\newblock {\em Ann. Sci. \'Ecole Norm. Sup. (4)}, 25(5):539--545, 1992.

\bibitem{C-L-Tsch}
A.~Chambert-Loir and Yu. Tschinkel.
\newblock Igusa integrals and volume asymptotics in analytic and adelic geometry.
\newblock {\em Confluentes Math.}, 2(3):351--429, 2010.

\bibitem{clt-vector}
A.~Chambert-Loir and Yu. Tschinkel.
\newblock Integral points of bounded height on partial equivariant compactifications of vector groups.
\newblock {\em Duke Math. J.}, 161(15):2799--2836, 2012.

\bibitem{Chen-Zhu}
Q.~Chen and Y.~Zhu.
\newblock {$\mathbb{A}^1$}-curves on log smooth varieties.
\newblock {\em J. Reine Angew. Math.}, 756:1--35, 2019.

\bibitem{C-TSKbook}
Jean-Louis Colliot-Th\'el\`ene and Alexei~N. Skorobogatov.
\newblock {\em The {B}rauer-{G}rothendieck group}, volume~71 of {\em Ergebnisse der Mathematik und ihrer Grenzgebiete. 3. Folge. A Series of Modern Surveys in Mathematics}.
\newblock Springer, Cham, [2021] \copyright 2021.

\bibitem{MR974910}
J.~Franke, Yu.~I. Manin, and Yu. Tschinkel.
\newblock Rational points of bounded height on {F}ano varieties.
\newblock {\em Invent. Math.}, 95(2):421--435, 1989.

\bibitem{FP}
C.~Frei and M.~Pieropan.
\newblock O-minimality on twisted universal torsors and {M}anin's conjecture over number fields.
\newblock {\em Ann. Sci. \'Ec. Norm. Sup\'er. (4)}, 49(4):757--811, 2016.

\bibitem{HuangIJNT}
Z.~Huang.
\newblock Distribution locale des points rationnels de hauteur born\'ee sur une surface de del {P}ezzo de degr\'e{} 6.
\newblock {\em Int. J. Number Theory}, 13(7):1895--1930, 2017.

\bibitem{HuangAA}
Z.~Huang.
\newblock Approximation diophantienne et distribution locale sur une surface torique.
\newblock {\em Acta Arith.}, 189(1):1--94, 2019.

\bibitem{HuangBSMF}
Z.~Huang.
\newblock Approximation diophantienne et distribution locale sur une surface torique {II}.
\newblock {\em Bull. Soc. Math. France}, 148(2):189--235, 2020.

\bibitem{HuangANT}
Z.~Huang.
\newblock Rational approximations on toric varieties.
\newblock {\em Algebra Number Theory}, 15(2):461--512, 2021.

\bibitem{KeelMcKernan}
S.~Keel and J.~McKernan.
\newblock Rational curves on quasi-projective surfaces.
\newblock {\em Mem. Amer. Math. Soc.}, 140(669):viii+153, 1999.

\bibitem{MR1158625}
J.~Koll\'ar, Y.~Miyaoka, and S.~Mori.
\newblock Rationally connected varieties.
\newblock {\em J. Algebraic Geom.}, 1(3):429--448, 1992.

\bibitem{MR4472281}
B.~Lehmann, A.~K. Sengupta, and S.~Tanimoto.
\newblock Geometric consistency of {M}anin's conjecture.
\newblock {\em Compos. Math.}, 158(6):1375--1427, 2022.

\bibitem{McK}
D.~McKinnon.
\newblock A conjecture on rational approximations to rational points.
\newblock {\em J. Algebraic Geom.}, 16(2):257--303, 2007.

\bibitem{McK-Roth}
D.~McKinnon and M.~Roth.
\newblock Seshadri constants, diophantine approximation, and {R}oth's theorem for arbitrary varieties.
\newblock {\em Invent. Math.}, 200(2):513--583, 2015.

\bibitem{MR4822120}
J.~Ortmann.
\newblock Integral points on a del {P}ezzo surface over imaginary quadratic fields.
\newblock {\em Res. Number Theory}, 10(4):Paper No. 90, 40, 2024.

\bibitem{Peyre}
E.~Peyre.
\newblock Beyond heights: slopes and distribution of rational points.
\newblock In {\em Arakelov geometry and {D}iophantine applications}, volume 2276 of {\em Lecture Notes in Math.}, pages 215--279. Springer, Cham, 2021.

\bibitem{Roth}
K.~F. Roth.
\newblock Rational approximations to algebraic numbers.
\newblock {\em Mathematika}, 2:1--20; corrigendum, 168, 1955.

\bibitem{Salberger}
P.~Salberger.
\newblock Tamagawa measures on universal torsors and points of bounded height on {F}ano varieties.
\newblock In E.~Peyre, editor, {\em Nombre et r\'epartition de points de hauteur born\'ee}, number 251 in Ast\'erisque, pages 91--258. Soc. Math. France, Paris, 1998.

\bibitem{Santens}
T.~Santens.
\newblock Manin's conjecture for integral points on toric varieties.
\newblock Preprint, arXiv:2312.13914v2, 2023.

\bibitem{serre-mordell}
J.-P. Serre.
\newblock {\em Lectures on the {M}ordell-{W}eil theorem}, volume E15 of {\em Aspects of Mathematics}.
\newblock Friedr. Vieweg \& Sohn, Braunschweig, 1989.

\bibitem{SerreGalois}
J.-P. Serre.
\newblock {\em Topics in {G}alois theory}.
\newblock Number~1 in Research Notes in Mathematics. A K Peters, Ltd., Wellesley, MA, 2 edition, 2008.

\bibitem{Siegel}
C.~L. Siegel.
\newblock {\em {\"U}ber einige {A}nwendungen diophantischer {A}pproximationen}, volume~1 of {\em Abhandlungen der Preußischen Akademie der Wissenschaften. Physikalisch-mathematische Klasse}.
\newblock Akademie-Verlag, Berlin, 1929.

\bibitem{Wilsch-toric}
F.~Wilsch.
\newblock Integral points of bounded height on a certain toric variety.
\newblock {\em Trans. Amer. Math. Soc. Ser. B}, 11:567--599, 2024.
}
\end{thebibliography}

\end{document}